\documentclass[10pt,reqno]{amsart}
\usepackage{amsmath}
\usepackage{amscd,amsthm,amsfonts,amsopn,amssymb,mathrsfs}

\newtheorem{theorem}{Theorem}[section]
\newtheorem{proposition}[theorem]{Proposition}
\newtheorem{corollary}[theorem]{Corollary}
\newtheorem{lemma}[theorem]{Lemma}
\newtheorem{remark}[theorem]{Remark}
\numberwithin{equation}{section}

\DeclareMathOperator{\supp}{supp\,}
\DeclareMathOperator{\sgn}{sgn\,}
\DeclareMathOperator{\loc}{loc}

\allowdisplaybreaks
\begin{document}
\title{Optimal function spaces for continuity of the Hessian determinant as a distribution}
\date{\today}
\thanks{The work of E.B. was supported by the National Science Foundation under Award No. DMS-1204557. The work of D.J. was supported by the Bergman Trust and 
NSF grant DMS-1069225.}
\author{Eric Baer}\address{Massachusetts Institute of Technology, 77 Massachusetts Ave., Cambridge, MA 02139--4381}\email{ebaer@math.mit.edu}
\author{David Jerison}\address{Massachusetts Institute of Technology, 77 Massachusetts Ave., Cambridge, MA 02139--4381}\email{jerison@math.mit.edu}

\begin{abstract}
We establish optimal continuity results for the action of the Hessian
determinant on spaces of Besov type into the space of distributions on
$\mathbb{R}^N$.  In particular, inspired by recent work of Brezis and
Nguyen on the distributional Jacobian determinant, we show that the
action is continuous on the Besov space of fractional order
$B(2-\frac{2}{N},N)$, and that all continuity results in this scale of
Besov spaces are consequences of this result.

A key ingredient in the argument is the characterization of
$B(2-\frac{2}{N},N)$ as the space of traces of functions in the
Sobolev space $W^{2,N}(\mathbb{R}^{N+2})$ on the subspace
$\mathbb{R}^N$ of codimension $2$.  The most delicate and elaborate
part of the analysis is the construction of a counterexample to
continuity in $B(2-\frac{2}{N},p)$ with $p>N$.
\end{abstract}

\maketitle

\section{Introduction}

Fix $N\geq 2$ and consider the class of scalar-valued functions $u$ on
$\mathbb{R}^N$.  The goal of this paper is to identify when the
Hessian determinant $\det(D^2u)$ makes sense as a distribution on
$\mathbb{R}^N$.  In the case $N=2$, it is a well known consequence of
integration by parts identities that the Hessian determinant is well
defined and continuous on $W^{1,2}(\mathbb{R}^2)$ (see \cite{I}).
For $N\geq 3$, spaces of integer order no longer suffice for optimal
results --- in fact, we will show below that $u\mapsto \det(D^2u)$ is
well defined and continuous from a function space of fractional order
$2-\frac{2}{N}$ into the space of distributions.

In particular, we consider the scale of Besov spaces on
$\mathbb{R}^N$, which we denote by $B(s,p)=B_s^{p,p}$, with norm
defined below in \eqref{eq1}, and we characterize the spaces in this
scale on which the Hessian determinant acts continuously.  Indeed,
continuity of the operator corresponds to a single master theorem in
the space $B(2-\frac{2}{N},N)$: the action is continuous on this
space, and consequently on every $B(s,p)$ satisfying $B(s,p)\subseteq
B_{\loc}(2-\frac{2}{N},N)$.\footnote{Here, $B_{\loc}(s,p)$ denotes the
  space of all $f$ such that $\chi f\in B(s,p)$ for all $\chi\in
  C_c^\infty(\mathbb{R}^N)$.}  Moreover, the action is not continuous
on any other space in the scale.

Our theorem is inspired by a recent theorem of Brezis and Nguyen
\cite{BN1} characterizing the spaces $B(s,p)$ of vector-valued maps
$f:\mathbb{R}^N\rightarrow\mathbb{R}^N$ on which the Jacobian
determinant $\det(Df)$ acts continuously (see also \cite{Mi}).  They
show that the Jacobian determinant is continuous from the space
$B(1-\frac{1}{N},N)$ into the space of distributions, and that
continuity fails for any space in this scale for which the inclusion
$B(s,p)\subseteq B_{\loc}(1-\frac{1}{N},N)$ does not hold (note that
in \cite{BN1} the Besov space $B(1-\frac{1}{N},N)$ is denoted by
$W^{1-\frac{1}{N},N}$).

We now give a formal statement of our main results.  For $1<s<2$ and $1\leq p< \infty$, let $B(s,p)$ be the function space defined via the norm
\begin{align}
\lVert u\rVert_{s,p}&:=\lVert u\rVert_{W^{1,p}}+ \bigg(\int_{\mathbb{R}^N}\int_{\mathbb{R}^N} \frac{|Du(x)-Du(y)|^p}{|x-y|^{N+\sigma p}}\,dxdy\bigg)^{1/p} \label{eq1}
\end{align}
where $\sigma=s-1$ satisfies $0<\sigma<1$.  Our first result establishes $B(2-\frac{2}{N},N)$ stability of the Hessian determinant.
\begin{theorem}
\label{thm1}
Fix $N\geq 3$.  Then for all $u_1, u_2, \varphi \in C_c^2(\mathbb{R}^N)$ one has
\begin{align}
\nonumber &\bigg|\int \Big(\det(D^2u_1)-\det(D^2u_2)\Big)\varphi\, dx\bigg|\\
&\hspace{0.4in}\leq C\lVert u_1-u_2\rVert_{2-\frac{2}{N},N}(\lVert u_1\rVert_{2-\frac{2}{N},N}^{N-1}+\lVert u_2\rVert_{2-\frac{2}{N},N}^{N-1})\lVert D^2\varphi\rVert_{L^\infty(\mathbb{R}^N)}.\label{est-1}
\end{align}
\end{theorem}

Standard approximation arguments give the following corollary, which
asserts existence of the Hessian determinant as a distribution for
functions in the Besov space $B(2-\frac{2}{N},N)$.
\begin{corollary}
\label{cor1}
The operator $u\mapsto
\det(D^2u):C_c^2(\mathbb{R}^N)\rightarrow\mathcal{D}'$ can be extended
uniquely as a continuous mapping denoted $u\mapsto \mathscr{H}(u)$
from the space $B(2-\tfrac{2}{N},N)$ to the space of distributions
$\mathcal{D}'(\mathbb{R}^N)$.  Moreoever, for $u_1,u_2\in
B(2-\frac{2}{N},N)$, the estimate \eqref{est-1} holds with the left
side replaced by
\begin{align*}
|\langle \mathscr{H}(u_1)-\mathscr{H}(u_2),\varphi\rangle|.
\end{align*}
\end{corollary}

As we mentioned above, the analogues of Theorem $\ref{thm1}$ and
Corollary $\ref{cor1}$ are well known in the case $N=2$, when the
regularity index becomes an integer, in which case the appropriate
function space is the usual Sobolev space $W^{1,2}$.

Having established our positive results concerning continuity of the Hessian determinant, we next address the question of optimality.

\begin{theorem}[Optimality in the scale $B(s,p)$]
\label{thm-opt}
Let $1<p<\infty$ and $1<s<2$ be such that $B(s,p)\not\subset B_{\loc}(2-\frac{2}{N},N)$.  Then there exist $u_k\in C_c^\infty(\mathbb{R}^N)$ and a test function $\varphi\in C_c^\infty(\mathbb{R}^N)$ such that
\begin{align}
\lVert u_k\rVert_{s,p}\rightarrow 0\quad \textrm{as}\,\, k\rightarrow\infty,\label{prop14-1}
\end{align}
and
\begin{align}
\int_{\mathbb{R}^N} \det(D^2u_k)\varphi\, dx\rightarrow\infty\quad \textrm{as}\,\, k\rightarrow\infty.\label{prop14-2}
\end{align}
\end{theorem}

Since complete characterizations of the parameters $s$ and $p$ for
which $B(s,p)\subset B_{\loc}(2-\frac{2}{N},N)$ are well known (see
Remark $\ref{rmk-inclusion}$ below), the assertions of Theorem
$\ref{thm1}$ and Theorem $\ref{thm-opt}$ correspond to an explicit
characterization of the continuity properties in the scale $B(s,p)$.

Having stated our main results, let us now describe the structure of
the arguments involved.  The proof of Theorem $\ref{thm1}$, given in
Section $\ref{sec-thm1}$, is inspired by \cite{BN1}, in
which the analogous theorem for the Jacobian determinant is proved
using an extension of $f:\mathbb{R}^N\rightarrow\mathbb{R}^N$ to a
function on $\mathbb{R}^{N+1}$ and a corresponding restriction or
trace theorem.  It turns out that the optimal results for the Hessian
require an extension of the scalar function $u$ to $\mathbb{R}^{N+2}$
(see Lemma $\ref{lem1}$).  Indeed, we can identify the space
$B(2-\frac{2}{N},N)$ as a natural candidate for a single ``master
function space'' by looking at known optimal continuity results in the
scale of spaces $W^{1,q}\cap W^{2,r}$ (see \cite{FM} and \cite{DM}).
In retrospect, one can also predict the numerology from the work of
P. Olver \cite{O} concerning higher order operators and integer-order
function spaces.

In analogy with \cite{BN1}, the statement of Theorem $\ref{thm1}$
immediately gives several corollaries; in particular, an appeal to
interpolation inequalities shows that continuity in
$B(2-\frac{2}{N},N)$ implies that the Hessian determinant, interpreted
as a distribution, is continuous in each of the spaces
$W^{1,q}(\mathbb{R}^N)\cap W^{2,r}(\mathbb{R}^N)$, with
$1<q,r<\infty$, $\frac{2}{q}+\frac{N-2}{r}=1$, $N\geq 3$.  Indeed,
such results were already known, with attention also paid to the cases
in which various notions of weak convergence suffice; see
\cite{DM,FM,I,DGG} and the references cited therein.  Note that
endpoint issues can be somewhat delicate: in particular, the result of
Theorem $\ref{thm1}$ does not establish continuity in the space
$W^{1,\infty}(\mathbb{R}^3)\cap W^{2,1}(\mathbb{R}^3)$ --- however,
the relevant continuity result does hold at this endpoint by arguments
of \cite{DM} and \cite{FM}.

We now turn to the proof of the $B(s,p)$ optimality result, Theorem
$\ref{thm-opt}$, which is contained in sections $\ref{sec-opt1}$,
$\ref{sec-prop2}$, and $\ref{section-5}$.  The analogous example of
failure of continuity in the work of Brezis-Nguyen is the result of an
elaborate construction: a sum of well-chosen atoms, scaled
at lacunary frequencies.  Our construction is even more
involved: we begin by identifying a suitable class of atoms of the
form
\begin{align*}
u(x)=x_N\prod_{i=1}^{N-1} g(x_i), \quad g:\mathbb{R}\rightarrow\mathbb{R},
\end{align*}
for which the Hessian determinant has a uniform sign on $\{x:x_N>0\}$
We then consider a sum of these
atoms rescaled at lacunary frequencies, and our task is to establish
blowup for the Hessian determinant of the sum in the sense of
distributions.  It is in this step that the essential complications
arise.  To estimate the Hessian determinant appropriately, we must
bound interactions between terms in the sum, and this means that we
must keep track of cancellation and reinforcement of Fourier modes in
$N$-multilinear expressions.  Indeed, to obtain the cancellations
required to complete the argument, we must use a lacunary sequence
that is much more sparse than exponential (see \eqref{eq-n-ell}).

\begin{remark}
\label{rmk-inclusion}
To interpret the assertions of Theorem $\ref{thm1}$, Corollary
$\ref{cor1}$, and Theorem $\ref{thm-opt}$, we recall the embedding
properties of the scale of spaces $B(s,p)$ (with $1<s<2$,
$1<p<\infty$) into the space $B_{\loc}(2-\frac{2}{N},N)$ (see
\cite{S2,Tr,BN1}):
\begin{enumerate}
\item[(a)] for $s+\frac{2}{N}>1+\max\{1,\frac{N}{p}\}$, the embedding
  $B(s,p)\subset B_{\loc}(2-\frac{2}{N},N)$ holds;
\item[(b)] for $s+\frac{2}{N}<1+\max\{1,\frac{N}{p}\}$, the embedding fails;
\item[(c)] for $s+\frac{2}{N}=1+\max\{1,\frac{N}{p}\}$, there are two sub-cases:
\begin{enumerate}
\item[(i)] if $p\leq N$, then the embedding $B(s,p)\subset B_{\loc}(2-\frac{2}{N},N)$ holds; while
\item[(ii)] if $p>N$, the embedding fails.
\end{enumerate}
\end{enumerate}
\end{remark}

In \cite{BN1}, Brezis and Nguyen obtain several additional results for
the Jacobian determinant (see also \cite{IM}).  In our analogous
context, we have obtained results which (i) recover classical weak
convergence results in the spaces $W^{1,q}\cap W^{2,r}$ for suitable
$1\leq q,r\leq\infty$ in a quantitative form, and (ii) address the
question of how weak the norm can be taken on the difference
$u_1-u_2$.  We plan to return to this issue in a subsequent work
\cite{BJ2}.

In \cite{O}, P. Olver considers higher-order notions of Jacobian
determinants and the associated minors, and studies stability
properties with respect to weak convergence on (integer-order) Sobolev
spaces for these operators.  Olver defines the notion of an $m$th
order Jacobian determinant of degree $r$ on the space $W^{m-\lfloor
  \frac{m}{r}\rfloor,\gamma}\cap W^{m-\lfloor
  \frac{m}{r}\rfloor-1,\delta}$, for suitable choices of $1\leq
\gamma,\delta\leq \infty$.  We expect that there are fractional
versions of his results, for instance that the $m$th order Jacobian
determinant of degree $r$ should be continuous from the space
$B(m-\frac{m}{r},r)$ into the space of distributions.  Other
generalized notions of the Hessian determinant are considered in
\cite{FM, F, Je, JJ, Ju}.

\subsection{Organization of the paper}

We conclude this introduction by outlining the structure of the
remainder of the paper.  In Section $2$, we give the proof of Theorem
$\ref{thm1}$, the positive result concerning distributional stability
of the Hessian determinant.  We then turn to the question of
optimality in Sections $3$, $4$ and $5$.  In Section $3$ we use
scaling arguments to establish Theorem $\ref{thm-opt}$ in the case
$p\geq N$ with $s+\frac{2}{N}<1+\frac{N}{p}$.  In Section $4$, we
establish Theorem $\ref{thm-opt}$ in the case $p>N$,
$s<2-\frac{2}{N}$ by constructing an explicit example, which we call an atom.
In Section $5$, which comprises the bulk of the paper, we establish
Theorem $\ref{thm-opt}$ in the remaining case $p>N$,
$s=2-\frac{2}{N}$, using a lacunary sum of atoms.  
The atoms are the same as in Section $4$.  In Appendix A, we recall 
some divergence and integration by parts formulas required in our 
proof.  In Appendix B we review the proof of the appropriate 
codimension two Besov extension lemma we require.  In Appendix C 
we give a more direct and explicit proof of the cancellations 
needed in the lacunary construction in dimension $N=3$.

\section{Proof of Theorem \ref{thm1}}
\label{sec-thm1}

We begin the proof with a lemma expressing the action of the Hessian 
determinant as a distribution in terms of an extension.
\begin{lemma}
\label{lem1}
Fix $N\geq 2$ and let $u,\varphi\in C_c^2(\mathbb{R}^N)$ be given.  Then there exists a collection 
\begin{align*}
(p_{i,j}:1\leq i,j\leq N+2)
\end{align*}
of homogeneous polynomials of degree $N$ in $(N+2)^2$ variables (arranged in an $(N+2)\times (N+2)$ matrix) such that for every pair of extensions 
\begin{align*}
U\in C_c^2(\mathbb{R}^N\times [0,1)\times [0,1))\quad \textrm{of}\quad u
\end{align*}
and 
\begin{align*}
\Phi\in C^2_c(\mathbb{R}^N\times [0,1)\times [0,1))\quad \textrm{of}\quad \varphi
\end{align*}
we have the identity
\begin{align}
\int_{\mathbb{R}^N} \det(D^2u)\varphi \,dx &=\sum_{i=1}^{N+2}\sum_{j=1}^{N+2} \int_{\mathbb{R}^{N+2}} p_{i,j}(D^2U(\widetilde{x}))(\partial_{i,j}\Phi)(\widetilde{x})d\widetilde{x}.\label{identity-1}
\end{align}
where $\widetilde{x}=(x,x_{N+1},x_{N+2})\in \mathbb{R}^{N+2}$.
\end{lemma}

\begin{proof} Denote $\partial_i = (\partial/\partial x_i)$ and row
vectors of length $N$, 
\begin{equation}\label{def-rmi}
\widehat{\nabla}_iU: = (\partial_1 U, \partial_2U, \dots, \partial_{i-1}U, 
\partial_{i+1} U, \dots, \partial_{N+1}U);
\quad R_m(i) : = \partial_m \widehat{\nabla}_iU
\end{equation}
The result of \cite[Lemma $3$]{BN1}, applied to the vector-valued function 
$\nabla_xu\in C^1(\Omega;\mathbb{R}^N)$ gives the identity
\begin{equation}\label{eq-10}
\int_{\Omega} \det(D^2u)(x)\varphi(x)\,dx
=\sum_{i=1}^{N+1} (-1)^{N-i}\int_{\Omega\times (0,1)} \det(R(i))(\partial_i\Phi)\Big|_{x_{N+2}=0} \,dxdx_{N+1},
\end{equation}
where, for each $\widetilde{x}\in\mathbb{R}^{N+2}$ and $i\in \{1,2,\cdots,N+1\}$,  
and the matrix $R(i)\in \mathbb{R}^{N\times N}$ has rows $R_m(i)$, $m=1,\dots,N$,
defined by \eqref{def-rmi}. (For completeness, \eqref{eq-10} is proved
in Appendix \ref{app-det-id}.)

Applying the fundamental theorem of calculus, we write the right-hand side 
of \eqref{eq-10} as
\begin{align}
\sum_{i=1}^{N+1}(-1)^{N+1-i}A(i),\label{eq-l1}
\end{align}
with
\begin{align*}
A(i)&:=\int_{\Omega\times (0,1)^2} \partial_{N+2}\Big[\det(R(i))(\partial_i\Phi)(\widetilde{x})\Big]d\widetilde{x}.
\end{align*}
Hence, 
\begin{align}
A(i)=\sum_{j=1}^{N} A_{j}(i)+\int_{\Omega\times (0,1)^2} \det(R(i))(\partial_{i,N+2}\Phi)d\widetilde{x}.\label{eq-positive-1}
\end{align}
where, for each $1\leq i\leq N+1$ and $1\leq j\leq N$, we have defined
\begin{align*}
A_{j}(i):=\int_{\Omega\times (0,1)^2} \det(B^*_{j}(i))(\partial_i\Phi)d\widetilde{x},
\end{align*}
with $B^*_j(i)$ as the $N\times N$ matrix having rows given by 
\begin{align*}
(B^*_j(i))_m&=\left\lbrace\begin{array}{cl}\partial_{N+2}R_j(i)&\textrm{if}\,\, m=j,\\
R_m(i)&\textrm{if}\,\, m\neq j,\end{array}\right.\quad \textrm{for}\quad 1\leq m\leq N.
\end{align*}

Fix $i\in \{1,\cdots,N+1\}$ and $j\in \{1,\cdots,N\}$.  Using  the integration by parts formula \eqref{det-parts-1} 
in the expression $A_{j}(i)$ with respect to the variable $x_j$, we therefore obtain
\begin{align*}
A_{j}(i)&=-\bigg(\sum_{\substack{k\in \{1,\cdots,N\}\\k\neq j}} \int_{\Omega\times (0,1)^2} \det(C^*_{j,k}(i))(\partial_i\Phi)d\widetilde{x}\bigg)\\
&\hspace{1.2in}-\int_{\Omega\times (0,1)^2} \det(D^*_j(i))(\partial_{i,j}\Phi)d\widetilde{x},
\end{align*}
where
$C^*_{j,k}(i)$
is the $N\times N$ matrix with rows given by
\begin{align*}
(C^*_{j,k}(i))_m=\left\lbrace\begin{array}{cl}\partial_{N+2}\widehat{\nabla}_iU&\textrm{if}\,\, m=j,\\
\partial_j R_k(i)&\textrm{if}\,\, m=k,\\
R_m(i)&\textrm{if}\,\, m\not\in \{j,k\},
\end{array}\right.\quad \textrm{for}\quad 1\leq m\leq N
\end{align*}
and where $D^*_j(i)$ is the $N\times N$ matrix with rows given by
\begin{align*}
(D^*_j(i))_m=\left\lbrace\begin{array}{cl}\partial_{N+2}\widehat{\nabla}_iU&\textrm{if}\,\, m=j\\
R_m(i)&\textrm{if}\,\, m\neq j
\end{array}\right.\quad \textrm{for}\quad 1\leq m\leq N.
\end{align*}

Taking the sum in $j$, we get
\begin{align}
\nonumber A(i)&=-\bigg(\sum_{\substack{(j,k)\in \{1,\cdots,N\}^2\\j\neq k}} \int_{\Omega\times (0,1)^2} \det(C^*_{j,k}(i))(\partial_i \Phi)d\widetilde{x}\bigg)  
 \\  
 \nonumber 
 &
\hspace{0.4in}
-\bigg(\sum_{j=1}^N \int_{\Omega\times (0,1)^2}\det(D^*_j(i))(\partial_{i,j}\Phi)d\widetilde{x}\bigg)  \\  
&\hspace{0.4in}
+\int_{\Omega\times (0,1)^2} \det(R(i))(\partial_{i,N+2}\Phi)d\widetilde{x}.
\label{eq-sum1}
\end{align}

Now, note that for each $(j,k)\in \{1,\cdots,N\}^2$ with $j\neq k$ the definition of $R_j(i)$ and $R_k(i)$ in \eqref{def-rmi} implies $\partial_jR_k(i)=\partial_kR_j(i)$.  Combining this equality with the alternating property of the determinant, we obtain 
\begin{align*}
\det(C^*_{j,k}(i))=-\det(C^*_{k,j}(i)), \quad j\neq k,
\end{align*}
and the first summation on the right-hand side of \eqref{eq-sum1} is equal to zero.

Taking the sum in $i$ (and recalling \eqref{eq-l1}, we have
\begin{align*}
\eqref{eq-10}&=\int_{\Omega\times (0,1)^2}\bigg[\sum_{i=1}^{N+1} (-1)^{N+1-i}\bigg\{\bigg(-\sum_{j=1}^N \det(D^*_j(i))(\partial_{i,j}\Phi)\bigg)\\
&\hspace{2.4in}+\det(R(i))(\partial_{i,N+2}\Phi)\bigg\}\bigg]d\widetilde{x}.
\end{align*}
The right-hand side has the desired form,
completing the proof of Lemma \ref{lem1}.
\end{proof}

Now we can finish the proof of  Theorem \ref{thm1}.  By a well known
theorem of E. M. Stein,
there is
a bounded linear extension operator
\begin{align*}
E:B^{2-\tfrac{2}{N},N}(\mathbb{R}^N)\rightarrow B^{2,N}(\mathbb{R}^N\times [0,1)^2).
\end{align*}
See Appendix \ref{app-b} for details.

Let $u_1, u_2\in C_c^2(\mathbb{R}^N)$ be given along with $\varphi\in C_c^2(\mathbb{R}^N)$, and set
\begin{align*}
U_i=Eu_i\quad\textrm{for}\quad i=1,2,\quad \textrm{and} \quad \Phi=E\varphi.
\end{align*}
We then have, using Proposition \ref{prop-extension}
in Appendix \ref{app-b},
\begin{align}
\lVert D^2U_i\rVert_{L^N(\mathbb{R}^N\times (0,1)\times (0,1))}&\leq C\lVert u_i\rVert_{2-\frac{2}{N},N},\quad i=1,2,\label{eq-extension-d-1}
\end{align}
and
\begin{align}
\lVert D^2\Phi\rVert_{L^\infty(\mathbb{R}^N\times (0,1)\times (0,1))}&\leq C\lVert D^2\varphi\rVert_{L^{\infty}(\mathbb{R}^N)}.\label{eq-extension-d-3}
\end{align}

Applying Lemma \ref{lem1} and invoking the  bound 
\begin{align*}
|p(A)-p(B)|&\leq C(|A|+|B|)^{N-1}|A-B|,\quad A,B\in\mathbb{R}^{(N+2)\times (N+2)},
\end{align*}
which is valid for any fixed homogeneous polynomial $p$ of degree $N$, we find
that the left-hand side of \eqref{est-1} is bounded by
\begin{align*}
&\frac{1}{2}\sum_{i,j=1}^{N+2}\int_{\mathbb{R}^N\times (0,1)\times (0,1)} \Big|p_{i,j}(D^2U_1)-p_{i,j}(D^2U_2)\Big|\, |\partial_{i,j}\Phi|\, d\widetilde{x}\\
&\hspace{0.2in}\leq \sum_{i,j=1}^{N+2}\int_{\mathbb{R}^N\times (0,1)\times (0,1)} C(|D^2U_1|+|D^2U_2|)^{N-1}|D^2(U_1-U_2)| \, |\partial_{i,j}\Phi|\, d\widetilde{x}.
\end{align*}
Estimating the right-hand side of the expression above using
H\"older's inequality and \eqref{eq-extension-d-1}--\eqref{eq-extension-d-3}, we obtain
\eqref{est-1}.   This completes the proof of Theorem $\ref{thm1}$.

\section{Optimality results I ($p\leq N$): scaling analysis.}
\label{sec-opt1}

In this section we establish the optimality results of Theorem $\ref{thm-opt}$ in the cases $p\leq N$.  By Remark \ref{rmk-inclusion}, the inclusion 
$B(s,p)\subset B_{\loc}(2-\frac{2}{N},N)$ fails whenever $s+\frac{2}{N}<1+\frac{N}{p}$.
Thus we shall see that continuous dependence fails for reasons of homogeneity.  
\begin{proposition}
Fix $N\geq 3$, $1<p\leq N$, and $0<s<2$ satisfying $s+\frac{2}{N}<1+\frac{N}{p}$.  Then there exist  $u_k \in  C_c^\infty(\mathbb{R}^N)$ and a test function $\varphi\in C_c^\infty(\mathbb{R}^N)$ satisfying the conditions \eqref{prop14-1} 
and \eqref{prop14-2} of Theorem \ref{thm-opt}.
\end{proposition}

\begin{proof}  Let $g\in C_c^\infty(\mathbb{R}^N)$ with $\supp
g\subset \{x\in\mathbb{R}^N:|x|<1\}$, and 
consider
$g_\epsilon:\mathbb{R}^N\rightarrow\mathbb{R}$ defined by
\begin{align*}
g_\epsilon(x)=\epsilon^{\sigma}g\left(\frac{x}{\epsilon}\right),\quad 0<\epsilon<1,
\end{align*}
where $\sigma>0$ is a fixed parameter specified later.
For $0<\epsilon<1$,
\begin{align*}
\lVert g_\epsilon\rVert_{s,p}\leq \lVert g_{\epsilon}\rVert_{L^p}^{1-\frac{s}{2}}\lVert D^2g_\epsilon\rVert_{L^p}^{\frac{s}{2}}=\epsilon^{\sigma+\frac{N}{p}-s}\lVert g\rVert_{L^p}^{1-\frac{s}{2}}\lVert D^2g\rVert_{L^p}^{\frac{s}{2}}
\end{align*}

To establish \eqref{prop14-2}, we start by showing that 
$g$ can be chosen so that, in addition,
\begin{align}
\int_{\mathbb{R}^N} \det(D^2g)|x|^2\,dx\neq 0\label{eq-g-condition}
\end{align}
Indeed, define 
\begin{align*}
g(x)=\int_0^{|x|} h(s)ds\,\,\textrm{for}\,\, x\in\mathbb{R}^N,
\end{align*}
for some $h\in C_c^\infty((0,1))$ satisfying   
\[
\int_0^1 h(r) dr = 0,  \ \quad \int_0^1 h(r)^N r \,dr \neq 0.
\]
The first condition implies that $g$ is compactly  supported in the unit ball.
Furthermore, $\det(D^2g) = (h(r)r)^{N-1} h'(r)$.  (See 
\cite{DM}, p. 59, where this identity was used for similar 
purposes.) Therefore,
\begin{equation*}
\int_{\mathbb{R}^N} \det(D^2g)|x|^2\, dx =C_N\int_{0}^\infty h(r)^{N-1}h'(r)r^2\, dr\\
=-\frac{2C_N}{N}\int_0^1 h(r)^Nr\,dr \neq0.
\end{equation*}
In other words, the second condition on $h$ implies \eqref{eq-g-condition}.

Let $\varphi\in C_c^\infty(\mathbb{R}^N)$ be such that
$\varphi(x) = |x|^2 + O(|x|^3)$ as $x\to 0$.  Then
\begin{align*}
\int_{\mathbb{R}^N} \det(D^2g_\epsilon)\varphi \,dx&=\epsilon^{(\sigma-2)N}\int_{\mathbb{R}^N} \det((D^2g)(x/\epsilon))\varphi(x)\,dx\\
&
=\epsilon^{(\sigma-1)N}\int_{\mathbb{R}^N} \det(D^2g(x))\varphi(\epsilon x)\,dx \\
& = \epsilon^{(\sigma-1)N + 2}\int_{\mathbb{R}^N} \det(D^2g(x))|x|^2 \,dx + 
O(\epsilon^{(\sigma-1)N + 3}).
\end{align*} 

Whenever $s+\frac{2}{N}<1+\frac{N}{p}$, we may choose $\sigma$ so that
\begin{align*}
s-\frac{N}{p}<\sigma<1-\frac{2}{N}.
\end{align*}
It follows that 
\begin{align*}
\lVert g_{\epsilon}\rVert_{s,p}\leq C\epsilon^{\sigma+\frac{N}{p}-s} \to 0
\quad \mbox{as} \ \epsilon\to 0
\end{align*}
and
\begin{align*}
\bigg|\int_{\mathbb{R}^N} \det(D^2g_\epsilon)\varphi \,dx\bigg|\geq c\epsilon^{(\sigma-1)N+2} \to \infty  \quad \mbox{as} \ \epsilon\to 0.
\end{align*}
\end{proof}

\section{Optimality results II ($p>N$, $s<2-\frac{2}{N}$): construction of atoms.}
\label{sec-prop2}

A fundamental tool in the rest of our analysis is a formula due to 
B.Y. Chen \cite{C}  for the Hessian determinant of functions 
$F:\mathbb{R}^N\rightarrow\mathbb{R}$ given as a tensor product
\begin{align*}
F(x)=\prod_{i=1}^N f_i(x_i),\quad x=(x_1,\cdots,x_N)\in\mathbb{R}^N.
\end{align*}
Writing $F(x)=\exp(\log(f_1)+\cdots +\log(f_N))$, an application of 
\cite[identity ($3.1$) on pg. 31]{C} gives 
\begin{align}
\det(D^2F)&=F(x)^{N-2}\bigg\{\bigg(\prod_{i=1}^{N} g_i(x_i)\bigg)
+\sum_{j=1}^N\bigg(\prod_{i\neq j} g_i(x_i)\bigg)[f'_j(x_j)]^2\bigg\}\label{hess-identity}
\end{align}
with
\begin{align*}
g_i(x)=f_i''(x)f_i(x)-[f_i'(x)]^2,\quad 1\leq i\leq N.
\end{align*}

\begin{proposition}
Fix $N\geq 3$, $N<p<\infty$, and $0<s<2-\frac{2}{N}$.  Let $\Omega\subset \{x\in \mathbb{R}^N:x_N>0\}$ be a nonempty open set, and let $\chi\in C_c^\infty(\mathbb{R}^N)$ be a smooth cutoff function with $\chi=1$ on $\Omega$.  For each $k\geq 1$, define $u_k:\mathbb{R}^N\rightarrow\mathbb{R}$ 
\begin{align*}
u_k=k^{-\alpha}\chi(x)x_N\prod_{i=1}^{N-1}\sin^2(kx_i)
\end{align*}
with $s<\alpha<2-\frac{2}{N}$ and $x=(x_1,\cdots,x_N)$.
Then for any $\Omega'\subset\subset\Omega$ and 
$\varphi\in C_c^\infty(\Omega)$ with $\varphi \ge0$, and $\varphi=1$ on $\Omega'$, 
the functions $u_k$ satisfy the conditions \eqref{prop14-1} and \eqref{prop14-2} of Theorem \ref{thm-opt}.
\end{proposition}

\begin{proof}
We begin by showing \eqref{prop14-1}.  This follows by writing
\begin{align*}
\lVert u_k\rVert_{s,p}&\leq C\lVert u_k\rVert_{L^p}^{1-\frac{s}{2}}\lVert u_k\rVert_{W^{2,p}}^{\frac{s}{2}}
\leq Ck^{s-\alpha},
\end{align*}
where the second inequality follows from the bounds 
\begin{align*}
\lVert u_k\rVert_{L^p}\leq C\lVert u_k\rVert_{L^\infty}\leq Ck^{-\alpha}\quad\textrm{and}\quad
\lVert D^2u_k\rVert_{L^p}\leq C\lVert D^2u_k\rVert_{L^\infty}\leq Ck^{2-\alpha}
\end{align*}
with constants depending on the measure of $\supp\chi$.  The desired convergence \eqref{prop14-1} 
now follows immediately from the condition $s<\alpha$.

On the other hand, by using \eqref{hess-identity} 
we obtain, for $x\in \Omega$,
\[
\det(D^2u_k)(x)
=-(-2)^{N}k^{2(N-1)-N\alpha}x_N^{N-2}\bigg(\prod_{i=1}^{N-1} \sin(kx_i)\bigg)^{2(N-1)}\bigg(\sum_{j=1}^{N-1} \cos^2(kx_j)\bigg).
\]
Thus $\det(D^2u_k)$ does not change sign in $\Omega$, and 
\begin{align}
\nonumber \bigg|\int_{\mathbb{R}^N} & \det(D^2u_k)\varphi \,dx\bigg|
\geq \bigg|\int_{\Omega'} \det(D^2u_k)\,dx\bigg|\\
&\hspace{0.2in}=k^{2N-2-N\alpha}\int_{\Omega'} x_N^{N-2}\bigg(\prod_{i=1}^{N-1}\sin(kx_i)\bigg)^{2(N-1)}\bigg(\sum_{j=1}^{N-1}\cos^2(kx_j)\bigg)\,dx.\label{eq-det-d2u_k}
\end{align}
Finally, the right-hand side of \eqref{eq-det-d2u_k} is bounded from below by a multiple 
of $k^{2N-2-N\alpha}$, and, by hypothesis, $2N-2-N\alpha>0$.  Therefore, we have 
\eqref{prop14-2}, completing the proof of the proposition.
\end{proof}

Note that if $N$ is even, the argument of this section works 
equally well with the simpler choice 
\begin{align*}
u_k=k^{-\alpha}x_N\prod_{i=1}^{N-1}\sin(kx_i),\quad k\geq 1,
\end{align*}
whose Hessian determinant is
\begin{align*}
\det(D^2u_k)=(-1)^{\frac{N}{2}}k^{2(N-1)-N\alpha}x_N^{N-2}\bigg(\prod_{i=1}^{N-1}\sin(kx_i)\bigg)^{N-2}\bigg(\sum_{i=1}^{N-1} \cos^2(kx_i)\bigg).
\end{align*}

\section{Optimality results III ($p>N$, $s=2-\frac{2}{N}$): interactions of atoms.}
\label{section-5}

In this section, we complete the proof of Theorem $\ref{thm-opt}$ by
establishing the result in the remaining case, when $p>N$ and
$s=2-\frac{2}{N}$.  In this case, we will need a highly lacunary sum
of atoms.

For each $k\geq 1$, fix $k\in\mathbb{N}$ with $k\geq 2$ and define 
$n_\ell$ by\footnote{Our methods can be tightened somewhat to yield 
slightly better growth rates, but even with sharper estimates in place 
the ordinary exponential growth of lacunary-type sequences of the 
form $n_{\ell} \sim C^\ell$ does not suffice.}
\begin{align}
n_{\ell}=k^{ (N^{3\ell}) }, \quad \ell = 1,\, 2, \, \dots  \label{eq-n-ell}
\end{align}
Define 
$w_k:\mathbb{R}^N\rightarrow\mathbb{R}$ by
\begin{align*}
w_k(x)=\sum_{\ell=1}^{k} \frac{1}{n_{\ell}^{2-(2/N)}\ell^{1/N}}g_\ell(x),\quad x\in\mathbb{R}^N,
\end{align*}
with $g_\ell:\mathbb{R}^N\rightarrow\mathbb{R}$ given by
\begin{align*}
g_{\ell}(x)=\prod_{i=1}^{N-1} \sin^2(n_{\ell}x_i),\quad x\in\mathbb{R}^N.
\end{align*}
Finally define $u_k:\mathbb{R}^N\rightarrow\mathbb{R}$ by
\begin{align}
u_k(x)=\chi(x)w_k(x)x_N\label{s3-uk}
\end{align}
for $x\in \mathbb{R}^N$, where $\chi\in C_c^\infty(\mathbb{R}^N)$ is a smooth cutoff function satisfying $\chi(x)=1$ for $x\in (0,2\pi)^N$.

\begin{proposition}
\label{prop3}
Fix $N\geq 3$, $N<p<\infty$, and $s=2-\frac{2}{N}$.  Then, there
exists $\varphi\in C_c^\infty(\mathbb{R}^N)$ with 
\begin{align*}
\varphi(x)=\prod_{i=1}^N \varphi_i(x_i),\quad x=(x_1,\cdots,x_N)\in\mathbb{R}^N,
\end{align*}
and $\supp \varphi\subset (0,2\pi)^N$, and there exist $c>0$ and
$K_0$  such that $u_k$ defined by \eqref{s3-uk} satisfies
\begin{align}
\sup_{k\in\mathbb{N}}\,\, \lVert u_k\rVert_{2-\frac{2}{N},p}<\infty\label{eq-prop3-goal1}
\end{align}
and
\begin{align}
\bigg|\int_{\mathbb{R}^{N}} \det(D^2u_k)(x)\varphi(x)\,dx\bigg|
\geq c(\log k )  - K_0,\quad k\geq 1.
\label{eq-prop3-goal2}
\end{align}
\end{proposition}
Note that the case $p>N$, $s=2-\frac{2}{N}$ of Theorem $\ref{thm-opt}$ follows
by setting
\begin{align*}
\widetilde{u}_k:=\frac{u_k}{\sqrt{\log(k)}}
\end{align*}
for any sequence $k\to \infty$. 
Furthermore, in light of Remark \ref{rmk-inclusion}, Proposition \ref{prop3} completes
the proof of the remaining cases of Theorem \ref{thm-opt}.

\begin{proof}[Proof of boundedness \eqref{eq-prop3-goal1} in Proposition \ref{prop3}]
The argument is nearly the same as the argument in \cite{BN1} and does
not require super-exponential lacunarity.  Standard estimates for
products in fractional order spaces show that it suffices to estimate
$\lVert w_k\rVert_{2-\frac{2}{N},p}$ on $\mathbb T^N$.  
Moreover, the Littlewood-Paley
characterization of the Besov space
$B^{2-\frac{2}{N},p}(\mathbb{T}^N)$ (see, e.g. \cite{Tr}) implies
\begin{align}
\lVert w\rVert_{2-\frac{2}{N},p}&\leq C\bigg(\lVert w\rVert_{L^p([0,2\pi]^N)}^p
+\sum_{j=1}^\infty 2^{(2-\frac{2}{N})jp}\lVert P_R(w) \rVert_{L^p([0,2\pi]^N)}^p\bigg)^{1/p}\label{eq-fracbd}
\end{align}
for suitable\footnote{The operators $P_j$ are defined by $$P_j(\sum
  \alpha_\ell e^{i\ell\cdot x})=\sum_{2^j\leq |\ell|<2^{j+1}}
  \bigg(\rho\Big(\frac{|\ell|}{2^{j+1}}\Big)-\rho\Big(\frac{|\ell|}{2^j}\Big)\bigg)\alpha_\ell
  e^{i\ell\cdot x}$$ for $j\geq 1$, where $\rho\in
  C_c^\infty(\mathbb{R})$ is a suitably chosen bump function.}
operators $P_j:L^p\rightarrow L^p$ such that there exists $C>0$ with
\begin{align*}
\lVert P_jf\rVert_{L^p}\leq C\lVert f\rVert_{L^p},\quad j\geq 1
\end{align*}
and
\begin{align*}
\supp(\widehat{P_jf})\subset \{n\in\mathbb{Z}^N:2^{j-1}<|n|<2^{j+2}\}, \quad 
j = 1,\, 2, \, \dots \, .
\end{align*}

We then have
\begin{align}
\lVert P_j(w_k)\rVert_{L^p}&\leq \sum_{\ell=1}^{k} \frac{1}{(n_\ell)^{2-(2/N)}\ell^{1/N}}\lVert P_j(g_{\ell})\rVert_{L^p}\label{eq-pruk}
\end{align}

To bound $P_jg_{\ell}$, write $g_{\ell}$ in exponential form,
\begin{align}
g_{\ell}(x)=\sum_{\varepsilon\in \{-1,0,1\}^{N-1}} a_{\varepsilon}e^{2n_{\ell}i\varepsilon\cdot \widehat{x}}\label{eq-sum-gell}
\end{align}
for $\widehat{x}=(x_1,x_2,\cdots,x_{N-1})\in\mathbb{R}^{N-1}$ and $x=(\widehat{x},x_N)\in\mathbb{R}^N$, and with $|a_\varepsilon|\leq 1$ for each $\varepsilon$.  Define
\begin{align*}
S(j,\ell)=\{\varepsilon\in \{-1,0,1\}^{N-1}:2^{j-1}\leq 2n_{\ell}|\varepsilon|< 2^{j+2}\},
\end{align*}
and 
\begin{align*}
\tilde\chi(j,\ell)= \begin{cases}
1 & \ \mbox{if} \ S(j,\ell) \neq \emptyset \\
0 & \ \mbox{if} \ S(j,\ell) = \emptyset 
\end{cases}
\end{align*}
Noting that there are at most $3^{N-1}$ terms in the summation in 
\eqref{eq-sum-gell},
\begin{equation} \label{eq-prf}
\lVert P_j(g_{\ell})\rVert_{L^p(\mathbb{T}^N)}\leq  C_N \tilde \chi(j,\ell)
\end{equation}

Next, observe that if $S(j,\ell) \neq \emptyset$, then
\begin{align}
\frac{2^{j-2}}{\sqrt{N-1}}\leq n_{\ell}<2^{j+1}.\label{eq-nlbound}
\end{align}
The lacunary property of the sequence $n_\ell$ implies that 
for each fixed $j$, $\tilde \chi(j,\ell) \neq0$ for at most one value of $\ell$.
Indeed, if $S(j,\ell_1) \neq \emptyset$ and $\ell_2 > \ell_1$, then
\[
\frac{2^{j-2}}{\sqrt{N}} < n_{\ell_1} 
\implies
n_{\ell_2} = k^{(N^{3\ell_2} - N^{3\ell_1})} n_{\ell_1} \ge 
2^{(N^6-N^3)} n_{\ell_1} \ge 8\sqrt{N}\,  n_{\ell_1} > 2^{j+1},
\]
which implies that $S(j,\ell_2) = \emptyset$.   Thus,  applying \eqref{eq-pruk},
\eqref{eq-prf}, 
and the fact that the sum has a single term at the  value of $\ell$ for which
$n_\ell$ is comparable to $2^j$, we have 
\begin{equation} \label{eq-singleterm}
\| P_j w_k\|_{L^p}^p
\le \left( \sum_{\ell = 1}^{k} \frac{C_N\tilde \chi(j,\ell)}{n_\ell^{(2-2/N)} \ell^{1/N} }\right)^p
\le C \sum_{\ell =1}^{k} 
\frac{\tilde \chi(j,\ell)}{2^{(2-2/N)jp} \ell^{p/N} }.
\end{equation}

Combining \eqref{eq-singleterm} with \eqref{eq-fracbd}, we find
\begin{align*}
\eqref{eq-fracbd}
&\leq C\bigg(\lVert w_k\rVert_{L_x^p}^p+\sum_{\ell=1}^{k}\frac{1}{\ell^{p/N}}\bigg(\sum_{j=1}^\infty \widetilde{\chi}(j,\ell)\bigg)\bigg)^{1/p}.
\end{align*}
The constraint  \eqref{eq-nlbound}  also implies that there are only
finitely many terms depending only on dimension
in the sum $\sum_j \tilde\chi(j,\ell)$ for each fixed $\ell$.  Moreover,
the sum over 
$\ell$ of $\ell^{-p/N}$ is bounded independent of $k$ because $p>N$.   Finally, 
the factor $1/n_\ell^{2-2/N}$ in the series defining $w_k$ shows
that $\|w_k\|_{L^p}$ is unformly bounded in $k$.
In all, \eqref{eq-fracbd} is bounded independent of $k$, and this completes
the proof of \eqref{eq-prop3-goal1}.
\end{proof}

Having established \eqref{eq-prop3-goal1}, it remains to
show \eqref{eq-prop3-goal2},  the distributional blow-up of
$\det(D^2u_k)$.   For notational convenience, we introduce
$f_\ell:\mathbb{R}^N\rightarrow\mathbb{R}$, given by
\begin{align}
f_{\ell}(x)=x_N g_\ell(x) = x_N\prod_{i=1}^{N-1} \sin^2(n_\ell x_i),\label{def-f-ell}
\end{align}
for $\ell\geq 1$.  Fix $k\geq 1$ and note that it follows from our hypotheses on the
cutoff function $\chi$ that 
\begin{align*}
u_k(x)=\sum_{\ell=1}^{k} \frac{1}{n_{\ell}^{2-(2/N)}\ell^{1/N}}f_{\ell}(x), \quad
x\in (0,2\pi)^N.
\end{align*}

Using multilinearity of the determinant we obtain
\begin{align}
\det(D^2u_k)=
\sum_{\boldsymbol \ell} C_{\boldsymbol \ell} \det (H_{\boldsymbol \ell} (x))
\label{eq-multilinear-det}
\end{align}
where the sum is over ${\boldsymbol \ell}=(\ell_1,\ell_2,\cdots,\ell_N)\in \{1,2,\cdots,k\}^N$, and we have set
\begin{align}
C_{{\boldsymbol \ell}}=\prod_{i=1}^N \frac{1}{n_{\ell_i}^{2-(2/N)}(\ell_i)^{1/N}}\label{eq-def-C}
\end{align}
and, for each $x\in\mathbb{R}^N$,  the $N\times N$ matrix $H_{{\boldsymbol \ell}}=H_{{\boldsymbol \ell}}(x)$ is given by
\begin{align}
H_{{\boldsymbol \ell}}=\Big( \partial_{i,j}f_{\ell_i} \Big)_{(i,j)\in \{1,\cdots,N\}^2}.\label{eq-def-D}
\end{align}

When ${\boldsymbol
  \ell}=(\ell,\ell,\cdots,\ell)$, $H_{{\boldsymbol\ell}}$ is the Hessian
  matrix of $f_\ell$.  On the other hand, when the indices
  of $\boldsymbol \ell$ are not equal, the matrix involves different
  functions in different rows.  We separate these two types
  of terms, and write
\begin{align*}
\det(D^2u_k)&=\sum_{\ell=1}^{k} C_{(\ell,\ell,\cdots,\ell)}\det(H_{(\ell,\ell,\cdots,\ell)})(x)\\
&\hspace{0.4in}+\sum_{{\boldsymbol \ell}\in \mathscr{L}} C_{{\boldsymbol\ell}}\det(H_{{\boldsymbol\ell}})(x)
\end{align*}
where $\mathscr{L}$ denotes the collection of all $N$-tuples ${\boldsymbol \ell}=(\ell_1,\cdots,\ell_N)\in \{1,2,\cdots,k\}^N$ such that there exist $i$, $j \in \{1,\cdots,N\}$ with $\ell_i\neq \ell_j$.  This gives
\begin{align}
\bigg|\int \det(D^2u_k)\varphi(x)\,dx\bigg|&\geq (I)-(II)\label{eq-det-phi-expn}
\end{align}
with
\begin{align}
(I)&:=\bigg|\sum_{\ell=1}^{k} C_{(\ell,\ell,\cdots,\ell)}\int \det(H_{(\ell,\ell,\cdots,\ell)})(x)\varphi(x)\,dx\bigg|\label{eq-def-I}
\end{align}
and
\begin{align}
(II)&:=\bigg|\sum_{{\boldsymbol\ell}\in\mathscr{L}} C_{{\boldsymbol\ell}}\int \det(H_{{\boldsymbol\ell}})(x)\varphi(x)\,dx\bigg|.\label{eq-def-II}
\end{align}

The term $(I)$ is the main term contributing to blowup on the right-hand side of \eqref{eq-det-phi-expn}, while $(II)$ will be interpreted as an error term.  Indeed, arguing as in Section $\ref{sec-prop2}$ above, the quantity $(I)$ is equal to
\begin{align*}
&\bigg|\sum_{\ell=1}^{k} \frac{1}{n_{\ell}^{2(N-1)}\ell} \int\det(D^2f_{\ell})\varphi(x)\,dx\bigg|\\
&\hspace{0.2in}=2^N\sum_{\ell=1}^{k} \frac{1}{\ell} \int x_N^{N-2}\bigg(\prod_{i=1}^{N-1} \sin(n_{\ell}x_i)\bigg)^{2(N-1)}\bigg(\sum_{j=1}^{N-1} \cos^2(n_{\ell}x_j)\bigg)\varphi(x)\,dx 
\end{align*}
and we therefore obtain
\begin{align}
(I)&\geq c\log(k)\label{eq-estimate-I}
\end{align}
for some $c>0$.

To estimate $(II)$, we will prove the following proposition.
\begin{proposition}
\label{prop-Hbound-3} Let $N\ge 3$.  Let the sequences $(C_{{\boldsymbol\ell}})$ and $(H_{{\boldsymbol\ell}})$ be as in \eqref{eq-def-C}--\eqref{eq-def-D}.  
Then there exists a dimensional constant $C>0$ such that 
for all ${\boldsymbol\ell}\in\mathscr{L}$ and $\varphi\in C_c^\infty((0,2\pi)^N)$
we have
\begin{align}
C_{{\boldsymbol\ell}}\bigg|\int_{\mathbb R^N} \det(H_{{\boldsymbol \ell}})(x)\varphi(x)\,dx\bigg|
\leq \frac{C\lVert \varphi\rVert_{C^{2}}}{k^{2N}}. \label{eq-Nd-goal}
\end{align}
\end{proposition}

To complete the proof of Proposition \ref{prop3}, note
that Proposition \ref{prop-Hbound-3}
implies
\begin{align*}
\nonumber
(II) \leq \sum_{{\boldsymbol\ell}\in\mathscr{L}} 
C_{{\boldsymbol\ell}}\bigg|\int \det(H_{{\boldsymbol\ell}})(x)\varphi(x)\,dx\bigg|
\nonumber 
\leq \sum_{{\boldsymbol\ell}\in \mathscr{L}} \frac{C \lVert \varphi
\rVert_{C^{2}}}{k^{2N}}
\leq C\lVert \varphi\rVert_{C^2}
\end{align*}
since $\mathscr{L}$ has fewer than $k^N$ elements. 
The desired conclusion \eqref{eq-prop3-goal2} now follows immediately
from \eqref{eq-det-phi-expn}.

All that remains is to prove Proposition \ref{prop-Hbound-3}.  The main tool we use 
is the Laplace expansion for the determinant.   To formulate it, we introduce some 
notations.

Denote by $C(m)$ the set of subsets of $ \{1,2,3,\dots, N\}$ with $m$ elements.
For an $N\times N$ matrix $H$ and $I$, $J\in C(m)$, 
\[
H^{(I,J)}
\]
denotes the $m\times m$ matrix whose rows are indexed by $I$ and
whose columns are indexed by $J$.  We also write $\# I$ for the
number of elements $m$ of $I$ and 
\[
\sigma(I) := \left(\sum_{i\in I} i\right) - \frac{m(m+1)}{2}
\]

\begin{lemma}[Laplace expansion]
\label{lem-laplace}
Let $H\in\mathbb{R}^{N\times N}$ be a square matrix.  Then for every $m\in\mathbb{N}$ and $I\in C(m)$, we have
\begin{align*}
\det(H)=(-1)^{\sigma(I)}\sum_{J\in C(m)} (-1)^{\sigma(J)}\det(H^{(I,J)})\det(H^{(I',J')}),
\end{align*}
where $I'=\{1,2,\cdots,N\}\setminus I$ and $J'=\{1,2,\cdots,N\}\setminus J$.
\end{lemma}
For a textbook treatment of Lemma $\ref{lem-laplace}$, see Section $3.7$ in \cite{E} (see also Section $I.2$ of \cite{MM}).

We isolate the rows corresponding to the largest frequency, as
follows.  For
a fixed $N$-tuple $\boldsymbol \ell = (\ell_1,\cdots,\ell_N)\in\mathscr{L}$, denote
\begin{align}
\ell_*=\max\{\ell_i:i=1,\cdots,N\}.\label{plus1}
\end{align}
Define
\begin{align}
I_{\boldsymbol \ell} :=\{i:\ell_i=\ell_*\}\subset \{1,2,\cdots,N\},\label{plus2}
\end{align}
We will abbreviate by $I=I_{\boldsymbol \ell}$ and $H=H_{\boldsymbol \ell}
$ when there is no potential for
confusion.   Lemma \ref{lem-laplace} applies to this matrix $H$ and index
set $I$ with $\#I = m$, giving
\begin{align}
\nonumber &\bigg|C_{\boldsymbol \ell}\int \det(H_{\boldsymbol\ell}(x))\varphi(x)\, dx\bigg|\\
&\hspace{0.2in}=\bigg|C_{\boldsymbol\ell}\sum_{J\in C(m)}\int (-1)^{\sigma(J)}\det(H^{(I,J)})\det(H^{(I',J')})\varphi\, dx\bigg|\label{eq-laplace-applied}
\end{align}

Note that since not all $\ell_i$ are equal, $m\le N-1$.  Furthermore, we
have separated the frequencies into two groups:
$\det(H^{(I,J)})$  involves only frequencies of the highest order $n_{\ell_*}$ 
or zero, while $\det(H^{(I',J')})$ involves only frequencies of lower order $n_{\ell_i}$ with $\ell_i\leq \ell_*-1$.  In fact, all of our estimates will be given in terms of $n_{\ell_*}$ and $n_{\ell_*-1}$.

We begin with the observation that for all $J\in C(m)$,
\begin{align*}
|\det(H^{(I,J)})|\leq 
Cn_{\ell_*}^{2m}\quad\textrm{and}\quad |\det(H^{(I',J')})|\leq Cn_{\ell_*-1}^{2(N-m)}.
\end{align*}
These estimates follow from degree considerations: each entry in the
$m\times m$ matrix $H^{(I,J)}$ is controlled by $n_{\ell_*}^2$, while
each entry of $H^{(I',J')}$ (a square matrix of size $N-m$) is
controlled by $n_{\ell_*-1}^2$.

Taking $J=I$, a stronger bound holds for $\det(H^{(I,I)})$ in the case
$N\in I$; in this case, one row and one column of $H^{(I,I)}$ each
consist of entries controlled by $n_{\ell_*}$, while all other entries
are controlled by $n_{\ell_*}^2$.  This gives
\begin{align}
|\det(H^{(I,I)})|\leq Cn_{\ell_*}^{2m-2} \quad \mbox{if} \ N\in I.\label{eq-HII-bd}
\end{align}

Since we have the bound $C_{\boldsymbol \ell}\leq n_{\ell_*}^{-(2-(2/N))m}$, 
\eqref{eq-HII-bd} and $m\le N-1$ lead to
\begin{align}
\nonumber 
\bigg| C_{\boldsymbol\ell}
\int \det(H^{(I,I)})\det(H^{(I',I')})\varphi\, dx \bigg|
&\leq C\lVert \varphi\rVert_{L^\infty}
\frac{(n_{\ell_*-1})^{2(N-m)}}{(n_{\ell_*})^{2-(2m/N)} }
\\
\nonumber 
&\leq C\lVert \varphi\rVert_{L^\infty}
\frac{(n_{\ell_*-1})^{2N}}
{(n_{\ell_*})^{2/N}}
\\
&\leq \frac{C\lVert \varphi\rVert_{L^\infty}}{k^{2N}}\label{eq-group1-bound},
\end{align}
where we have used the definition of $n_\ell$ given 
in \eqref{eq-n-ell} to obtain the last inequality.

This estimates the contribution of the term $J=I$ to \eqref{eq-laplace-applied} 
in the case $N\in I$.  To handle the contribution of the remaining terms, we are 
required to identify additional cancellations in $\det(H^{I,J})$.  Before proceeding, 
we condense notations further.  When $\boldsymbol \ell$ and $\ell_*$ are fixed,  
we will use the abbreviation
\begin{align*}
n=n_{\ell_*}.
\end{align*}
Moreover, with this convention in mind and in view of the particular structure of 
the matrix $H$, it will be useful to introduce the notation
\begin{align}
\widehat{S}_{A}=\prod_{i\in \{1,2,\cdots,N-1\}\setminus A} \sin^2(nx_i)\,\, \textrm{for}\,\, A\subset\{1,2,\cdots,N-1\},\label{def-SA}
\end{align}
and use the abbreviation $\widehat{S}_{a_1,a_2,\cdots,a_k}=\widehat{S}_{\{a_1,a_2,\cdots,a_k\}}$.

\begin{lemma}
\label{lem-2} Fix $\boldsymbol \ell \in \mathscr{L}$ and suppose that $I= I_{\boldsymbol\ell}\in C(m)$ is 
as in \eqref{plus2}.  Then for each $J\in C(m)$ with
\begin{align*}
\#(I\cap J)\leq m-2
\end{align*}
we have
\begin{align}
\det(H^{(I,J)})=0.\label{eq-det-goal}
\end{align}
\end{lemma}

\begin{proof}
We exhibit a dependence relation between two rows in $H^{(I,J)}$ corresponding to distinct indices $i_1, i_2 \in I\setminus J$.  In particular, let $i_1,i_2\in I\setminus J$ be given with $i_1\neq i_2$, set $f:=f_n=f_{n_{\ell_*}}$, and, for and $x\in \mathbb{R}^N$, let 
\begin{align*}
v_k(x)=\Big(\partial_{i_k,j}f\Big)_{j\in J}\in\mathbb{R}^{m},\quad k=1,2,
\end{align*}
denote the rows of $H^{(I,J)}$ corresponding to indices $i_1$ and $i_2$.  We consider three cases:

\vspace{0.1in}

\noindent \underline{Case 1}: $N\not \in \{i_1,i_2\}\cup J$.

\vspace{0.1in}

In this case, we obtain
\begin{align}
v_1(x)&=\Big(n^2x_N\sin(2nx_{i_1})\sin(2nx_{j})\widehat{S}_{i_1,j}\Big)_{j\in J}\label{eq-v1}
\end{align}  
and
\begin{align}
v_2(x)&=\Big(n^2x_N\sin(2nx_{i_2})\sin(2nx_{j})\widehat{S}_{i_2,j}\Big)_{j\in J}.\label{eq-v2}
\end{align} 

Direct calculation now shows that we have the identity
\begin{align}
\Big(\sin(2nx_{i_2})\sin^2(nx_{i_1})\Big)v_1(x)-\Big(\sin(2nx_{i_1})\sin^2(nx_{i_2})\Big)v_2(x)=0\label{eq-identity-1}
\end{align}
which immediately gives \eqref{eq-det-goal}, completing the proof of the Lemma in Case $1$.

\vspace{0.1in}

\noindent \underline{Case $2$}:  $N\in \{i_1,i_2\}$.

\vspace{0.1in}

In this case, we have $N\not\in J$.  Assume without loss of generality that $i_1=N$.  We then get the identity
\begin{align*}
v_1(x)&=\Big(n\sin(2nx_{j})\widehat{S}_{j}\Big)_{j\in J}.
\end{align*}
Moreover, since $i_2\neq N$ (recall that $i_1$ and $i_2$ are distinct), $v_2$ satisfies the equality given in \eqref{eq-v2}.  This leads to
\begin{align*}
\Big(nx_N\sin(2nx_{i_2})\Big)v_1(x)-\Big(\sin^2(nx_{i_2})\Big)v_2(x)=0,
\end{align*}
which completes the proof of the lemma in this case.

\vspace{0.1in}

\noindent \underline{Case $3$}: $N\in J$.

\vspace{0.1in}

In this case, we note that $i_1,i_2\in I\setminus J$ implies $i_1\neq N$ and $i_2\neq N$.  Assume without loss of generality that $a_m=N$.  We then obtain
\begin{align*}
v_k(x)&=\Big(v_{k,j}(x)\Big)_{j\in J}\quad\textrm{for}\quad k=1,2,
\end{align*}
where, for each $j\in J$, the quantity $v_{k,j}(x)$ is given by
\begin{align*}
v_{k,j}(x)=\left\lbrace\begin{array}{ll}n^2x_N\sin(2nx_{i_k})\sin(2nx_{j})\widehat{S}_{i_k,j},&\quad\textrm{if}\,\, j\neq N,\\
n\sin(2nx_{i_k})\widehat{S}_{i_k},&\quad\textrm{if}\,\, j=N.\end{array}\right.
\end{align*}
The identity \eqref{eq-identity-1} now follows as in Case $1$ above, which completes the proof of Lemma $\ref{lem-2}$ in Case $3$.
\end{proof}

Having shown Lemma \ref{lem-2}, it remains to estimate the contributions of $J\in C(m)$ to \eqref{eq-laplace-applied} with $\# (I\cap J) =m$ and $\#(I\cap J) = m-1$. 

\begin{remark}
\label{rmk1}
For all $J\in C(m)$ there exists an integer $\alpha\leq m$ and a sequence of coefficients $(c_z)\subset\mathbb{C}$ such that
\begin{align}
\det(H^{(I,J)})=\sum_{z\in \Lambda} c_z e^{2n_{\ell_*}iz\cdot \widehat{x}} (x_N)^{\alpha},\label{eq-lem3-det-expansion}
\end{align}
with $\widehat{x}=(x_1,x_2,\cdots,x_{N-1})$ and
\begin{align}
\Lambda=\bigg\{z\in \mathbb{Z}^{N-1}: |z_i|\leq N-1 \,\, \textrm{for}\,\, i=1,\cdots,N-1\bigg\},\label{def-Z}
\end{align} 
and with
\begin{align}
|c_z|\leq Cn_{\ell_*}^{2m}\,\, \textrm{for all}\,\, z\in \Lambda.\label{eq-bound-cz}
\end{align}
\end{remark}

\begin{proof}  Let $J\in C(m)$, and fix $1\le j \le N-1$. 
Each entry in the matrix $H^{(I,J)}$
is a polynomial of degree $1$ in the pair of expressions $e^{\pm 2 inx_j}$.
It follows that for every such $j$, $\det(H^{(I,J)})$
is a polynomial of degree at most $m$ in $e^{\pm 2 inx_j}$.
Because $m\le N-1$, we obtain a sum over the set $\Lambda$. 

We can likewise characterize the degree of each entry of
$\det(H^{(I,J)})$ in the variable $x_N$.  The full matrix
$H$ has the factor $x_N$ in each entry except in the last
row and column, where there is no factor of $x_N$.  It
follows that $x_N$ appears in $\det(H^{I,J})$ to the power 
$\alpha= m-2$, $m-1$, or $m$, depending on whether the $N$th column
is present in both of $I$ and $J$, one of them, or neither.

Finally, since each entry in $\det(H^{(I,J)})$ is controlled by
$(n_{\ell_*})^2$, one obtains $|\det(H^{(I,J)})|\leq
(n_{\ell_*})^{2m}$.  This implies $|c_z|\leq C(n_{\ell_*})^{2m}$ as
desired.
\end{proof}

The next lemma shows that the constant Fourier coefficient vanishes for 
the remaining contributions to \eqref{eq-laplace-applied}.

\begin{lemma}
\label{lem-3}
Let $N$, $I$ and $m$ be as above, and let $J\in C(m)$ be given.  If either of the conditions
\begin{enumerate}
\item[(a)] $I=J$ and $N\not\in I$, or
\item[(b)] $\#(I\cap J)=m-1$,
\end{enumerate}
are satisfied, then $c_{(0,0,\cdots,0)}=0$.
\end{lemma}

\begin{proof}
We will show that each of the conditions (a) and (b) imply the equality
\begin{align}
\int_{[0,2\pi]^{N-1}} \det(H^{(I,J)}) d\widehat{x}=0.\label{eq-f0-goal}
\end{align}

Suppose that condition (a) holds, i.e. $I=J$ and $N\not\in I\cap J=I$.  We then have
\begin{align*}
\det(H^{(I,J)})=\det(H^{(I,I)})=(\widehat{S}_Ix_N)^{m}\det(G),
\end{align*}
where $G$ is the Hessian matrix of the function $h:\mathbb{R}^m\rightarrow\mathbb{R}$ defined on the $m$ variables $x_I=(x_{i})_{i\in I}$ by
\begin{align*}
h(x_{I})&=\widehat{S}_{I'}=\prod_{i\in I}\sin^2(nx_{i}).
\end{align*}

Using the formula \eqref{hess-identity}, we therefore obtain
\begin{align*}
\det(G)=(-2n^2)^m\bigg(\prod_{i\in I} \sin(nx_{i})\bigg)^{2(m-1)}\bigg(1-2\sum_{j\in I} \cos^2(nx_{j})\bigg).
\end{align*}

To conclude the argument in this case, we compute
\begin{align}
\int_{[0,2\pi]^m} \bigg(\prod_{i\in I} \sin(nx_{i})\bigg)^{2(m-1)}\bigg(1-2\sum_{j\in I} \cos^2(nx_{j})\bigg) \,dx_{I}=0.\label{eq-f0-1}
\end{align}
Indeed, \eqref{eq-f0-1} follows immediately from the multiplicative structure of the integrand and the identity
\begin{align*}
\int_0^{2\pi} \sin^{2(m-1)}(nx)\cos^2(nx)\,dx&=\frac{1}{2m}\int_0^{2\pi} \sin^{2(m-1)}(nx)\,dx,
\end{align*}
which is the result of an elementary computation using integration by parts.  The equality \eqref{eq-f0-goal} therefore holds as desired.

We now turn to the case of condition (b).  Let $i_*$ be the singleton element of $I\setminus J$ and $j_*$ be the singleton element of $J\setminus I$.  It follows from $i_*\neq j_*$ that at least one of the two indices $i_*$ and $j_*$ is not equal to $N$.  Suppose without loss of generality that $j_*\neq N$.  We may then write, for some sequence $(\sigma_i)\subset \{-1,1\}$,
\begin{align*}
\det(H^{(I,J)})&=\sum_{i\in I} \sigma_ih_i
\end{align*}
with
\begin{align*}
h_i:=(\partial_{i,j_*}f_{\ell_*})\det(H^{(I\setminus\{i\},J\setminus \{j_*\})}),\quad i\in I.
\end{align*}

Note that we have $i\neq j_*$ for all $i\in I$.  Then, recalling that $j_*\neq N$ by hypothesis, for each $i\in I$ we have
\begin{align*}
\det(H^{(I\setminus \{i\},J\setminus \{j_*\})})=\sin^{2m-2}(nx_{j_*})g(x)
\end{align*}
for some function $g:\mathbb{R}^N\rightarrow \mathbb{R}$ independent of the variable $x_{j_*}$.  

Recalling the definition \eqref{def-f-ell} of $f_{\ell_*}$ we therefore have, for every $i=1,\cdots,m$, 
\begin{align*}
h_i&=\Big(\sin(2nx_{j_*})\sin^{2m-2}(nx_{j_*})\Big)\widetilde{g}(x)
\end{align*}
where $\widetilde{g}:\mathbb{R}^N\rightarrow\mathbb{R}$ is independent of the variable $x_{j_*}$.  It now follows that each $h_i$ is an odd function in the variable $x_{j_*}$, so that since $i\in I$ was arbitrary, the equality \eqref{eq-f0-goal} is satisfied.  This completes the proof of Lemma $\ref{lem-3}$.
\end{proof}

Having established Lemma $\ref{lem-2}$ and Lemma $\ref{lem-3}$, we can
now conclude the proof of Proposition $\ref{prop-Hbound-3}$.  In view of
\ref{lem-3}, all remaining contributions to \eqref{eq-laplace-applied} can
be estimated by integration by parts.

\begin{proof}[Proof of Proposition $\ref{prop-Hbound-3}$] 

Recall that we have fixed $k$, and that
$\mathscr{L}$ is the set of all $N$-tuples $(\ell_1,\cdots,\ell_N)\in \{1,\cdots,k\}^N$ 
with not all of the $\ell_i$ equal.  We have also fixed 
${\boldsymbol \ell}=(\ell_1,\ell_2,\cdots,\ell_N)\in\mathscr{L}$, defined $\ell_*$ and 
$I= I_{\boldsymbol \ell}$ in \eqref{plus1} and \eqref{plus2}, and set $m=\# I$.  
With this notation in hand, we have
\begin{align*}
\eqref{eq-laplace-applied}&\leq \bigg(\prod_{i=1}^N \frac{1}{n_{\ell_i}^{2-(2/N)}}\bigg)\sum_{J\in C(m)} \Big|\int \det(H^{(I,J)})\det(H^{(I',J')})\varphi \,dx\Big|
\end{align*}
Collecting \eqref{eq-group1-bound}, Lemma \ref{lem-2}, Remark \ref{rmk1}, and Lemma \ref{lem-3}, the right-hand side of this expression is bounded by 
\begin{align}
\frac{C\lVert \varphi\rVert_{L^\infty}}{k^{2N}}+C(n_{\ell_*})^{2m/N}\bigg(\sup_{\substack{J\in C(m)\\z\in \Lambda\setminus\{0\}}} \bigg|\int  e^{2n_{\ell_*}iz\cdot \widehat{x}}(x_N)^{\alpha}\det(H^{(I',J')})\varphi \,dx\bigg|\bigg).\label{eq-5-26}
\end{align}
where $C>0$ is a dimensional constant, and where we have again used the bound $C_{\boldsymbol\ell}\leq C(n_{\ell_*})^{-(2-(2/N))m}$.

We estimate each of the integrals in the supremum.  Let $z=(z_1,\cdots,z_{N-1})\in \Lambda\setminus\{0\}$ be given, and choose $j\in \{1,\cdots,N-1\}$ such that $z_j\neq 0$.  Integrating by parts two times in the $x_j$ variable, we obtain the identity
\begin{align}
\nonumber &\int e^{2n_{\ell_*}iz\cdot \widehat{x}}(x_N)^\alpha\det(H^{(I',J')})\varphi \,dx\\
&\hspace{0.4in}=-\frac{1}{4(n_{\ell_*})^2(z_j)^2}\int e^{2n_{\ell_*}iz\cdot \widehat{x}}(x_N)^\alpha\partial_{j}^2\Big[\det(H^{(I',J')})\varphi\Big] \,dx.\label{eq-zterm}
\end{align}
To bound the right-hand side of \eqref{eq-zterm}, we note that 
\begin{align*}
|\partial_{x_j}^2\det(H^{(I',J')})|&\leq C(n_{\ell_*-1})^{2(N-m)+2}\leq C(n_{\ell_*-1})^{2N}
\end{align*}
(here and elsewhere, $C$ is a dimensional constant, which may vary from line to line).
It follows that
\begin{align}
|\eqref{eq-zterm}|\leq \frac{C(n_{\ell_*-1})^{2N}}{(n_{\ell_*})^2}\label{eq-ztermbound}
\end{align}
This estimate then leads to
\begin{align*}
\eqref{eq-5-26}&\leq \frac{C\lVert \varphi\rVert_{L^\infty}}{k^{2N}}+\frac{C(n_{\ell_*-1})^{2N}}{(n_{\ell_*})^{2-2m/N}}\leq \frac{C\lVert \varphi\rVert_{L^\infty}}{k^{2N}}+\frac{C(n_{\ell_*-1})^{2N}}{(n_{\ell_*})^{2/N}}\leq \frac{C\lVert \varphi\rVert_{C^2}}{k^{2N}},
\end{align*}
completing the proof of Proposition \ref{prop-Hbound-3} and hence of Proposition \ref{prop3} and Theorem $\ref{thm-opt}$.
\end{proof}

\appendix

\section{Divergence identities}

\label{app-det-id}

In this appendix, we review some well known identities for the Jacobian and Hessian determinants.  We begin by recalling two formulas involving derivatives of the determinant.  The first formula is an instance of the product rule: given $N$ functions $f_1,f_2,\cdots,f_N\in C^1(\mathbb{R}^N;\mathbb{R}^N)$, one has
\begin{align}
\partial_{x_i}\det(f_1,f_2,\cdots,f_N)=\sum_{j=1}^{N} \det(f_1,\cdots,f_{j-1},\partial_{x_i}f_j,f_{j+1},\cdots,f_N)\label{det-product-1}
\end{align}
for each $i\in \{1,\cdots,N\}$.  

The second formula is a closely related expression for integration by parts.  Fix $i\in \{1,\cdots,N\}$.  Multiplying both sides of \eqref{det-product-1} by $\psi\in C_c^1(\mathbb{R}^N)$, integrating in the $x_i$ variable, and using integration by parts, we obtain
\begin{align}
\nonumber &-\int_{\mathbb{R}} \det(f_1,\cdots,f_N)(\partial_{x_i}\psi) \,dx_i\\
&\hspace{0.6in}= \sum_{j=1}^{N}\int_{\mathbb{R}} \det(f_1,f_2,\cdots,f_{j-1},\partial_{x_i}f_j,f_{j+1},\cdots,f_N)\psi \,dx_i.\label{det-parts-1}
\end{align}

Next, recall that if $f = (f_1,\dots,f_k):\mathbb{R}^k\rightarrow\mathbb{R}^k$,
then
\begin{align}
\det(Df)\,dx_1\wedge dx_2\wedge\cdots\wedge dx_k&=df_1\wedge df_2\wedge \cdots \wedge df_k\label{eq-df-0}\\
&=(-1)^{i+1}d\Big(f_i \,(df_1\wedge \cdots \wedge \widehat{(df_i)}\wedge \cdots \wedge df_{k})\Big)\label{eq-df-1}
\end{align}
for each $i=1,\cdots,k$, in which the second equality follows from the product/Leibniz rule for differential forms and $d^2=0$.  Integrating \eqref{eq-df-1} against a test function $\phi\in C^\infty_c(\mathbb{R}^k)$ then gives
\begin{align}
\nonumber \int_{\mathbb{R}^k} \det(Df)\phi \,dx&=(-1)^{i+1}\int_{\mathbb{R}^k} \phi d\Big(f_i\, (df_1\wedge \cdots\wedge \widehat{(df_i)}\wedge\cdots \wedge df_k)\Big)\\
&=(-1)^i\int_{\mathbb{R}^k} f_i\, d\phi\wedge df_1\wedge\cdots\wedge \widehat{(df_i)}\wedge \cdots\wedge df_k.\label{eq-app-1-1}
\end{align}

\subsection*{Morrey's identity \cite[\S4.4]{Mo}}
Let $f:\mathbb{R}^k\rightarrow\mathbb{R}^k$, then 
\begin{align}
\sum_{j=1}^k \partial_j[C_{ij}(Df)]=0\label{eq-app-a-cofactor}
\end{align}
for all $i=1,\cdots,k$, where $C_{ij} = C_{ij}(Df)$ denotes the $(i,j)$th cofactor of the matrix $Df\in \mathbb{R}^{k\times k}$ with entries $(Df)_{ij}=\partial_jf_i$.  

In the interest of completeness, we record the compact proof of this identity, given in \cite[Lemma $2.15$]{AFP}.  Fix, $i$ and define
\begin{align*}
\eta_i=df_1\wedge \cdots \wedge df_{i-1}\wedge df_{i+1}\wedge\cdots df_k,
\end{align*}
and
\begin{align*}
\omega_j=dx_1\wedge \cdots \wedge dx_{j-1}\wedge dx_{j+1}\wedge \cdots \wedge dx_k,\quad 1\leq j\leq k.
\end{align*}
Then, 
\begin{align*}
\eta_i=\sum_{j=1}^k (-1)^{i+j}C_{ij}\omega_j.
\end{align*}
Taking the exterior derivative of both sides (and noting that $d^2=0$ implies $d\eta_i=0$, as well as $d\omega_j=0$ for all $1\leq j\leq k$), we obtain
\begin{align*}
0=\sum_{j=1}^k (-1)^{i+j}dC_{ij}\wedge \omega_j&=\sum_{j=1}^k (-1)^{i+j}(\partial_jC_{ij})\, dx_j\wedge \omega_j\\
&=(-1)^i\sum_{j=1}^k (\partial_jC_{ij})\, dx_1\wedge \cdots \wedge dx_k.
\end{align*}
giving \eqref{eq-app-a-cofactor} as desired.

\subsection*{The Brezis-Nguyen extension identity}

We now use the discussion of the previous section to establish the identity \eqref{eq-10}.  Recall that this is precisely the statement of \cite[Lemma $3$]{BN1}; we repeat their proof here for completeness.  

Let $u\in C_c^2(\mathbb{R}^N)$ and let $U$ be its extension to $\mathbb{R}^{N+1}$ as in the statement of Lemma $\ref{lem1}$.  Set $g_i=\partial_{i}u$ and $G_i=\partial_iU$ for $i=1,\cdots,N$.  Then
\begin{align*}
\int_{\mathbb{R}^N} \det(Dg)(x)\varphi(x)\,dx&=-\int_{\mathbb{R}^N\times (0,1)} \partial_{x_{N+1}}[\det(D_xG)\Psi]\,dxdx_{N+1},
\end{align*}
and
\begin{align*}
\partial_{x_{N+1}}[\det(DG)\Psi]=\partial_{N+1}\det(DG)\Psi+\det(DG)\partial_{N+1}\Psi.
\end{align*}
Moreover, defining $F:\mathbb{R}^{N+1}\rightarrow \mathbb{R}^{N+1}$ by $F(x)=(G_1(x),\cdots,G_N(x),1)$ and applying \eqref{eq-app-a-cofactor} with $i=k=N+1$, one obtains
\begin{align*}
\partial_{N+1}\det(DG)&=\partial_{N+1}C_{N+1,N+1}(DF)\\
&\hspace{0.4in}=-\sum_{j=1}^{N} \partial_j[C_{N+1,j}(DF)]=\sum_{j=1}^N (-1)^{N-j}\partial_j\det(R(j)),
\end{align*}
where we have defined $R(j):\mathbb{R}^{N+1}\rightarrow\mathbb{R}^{N\times N}$, $1\leq j\leq N$, in terms of its rows as $(R(j))_m=\widehat{\nabla}_j G_m$ for $1\leq m\leq N$, using the notation $\widehat{\nabla}_i$ established in the proof of Lemma $\ref{lem1}$ (see \eqref{def-rmi} and the surrounding discussion).  Subsequent applications of integration by parts on each term then establish \eqref{eq-10} as desired.

\subsection*{Hessian identity}
We conclude this appendix by giving an integration by parts identity for the Hessian determinant.  Let $h,\varphi\in C^\infty_c(\mathbb{R}^N)$ be given.  Then
\begin{align}
\nonumber &\int_{\mathbb{R}^N} \det(D^2h)\varphi \,dx\\
\nonumber &\hspace{0.2in}=-\frac{1}{N!}\sum_{\sigma,\tau\in S_N} \sgn(\sigma)\sgn(\tau)\bigg(\int_{\mathbb{R}^N} (\partial_{\sigma(1)}h)(\partial_{\tau(1)}h)(\partial_{\sigma(2),\tau(2)}\varphi)\\
&\hspace{2.4in} \cdot(\partial_{\sigma(3),\tau(3)}h)\cdots (\partial_{\sigma(N),\tau(N)}h)\,dx\bigg).\label{parts-det-hessian}
\end{align}
Taking $f=\nabla h$, this identity is proved from \eqref{eq-app-1-1} by a second application of integration by parts, combined with symmetry considerations which lead to the ``canonicalized'' form given here.  It was used to define an earlier generalized notion of the Hessian determinant operator in \cite{I} (in integer-order Sobolev spaces); see also \cite{FM} and the references cited there.

\section{Boundedness of the extension operator}
\label{app-b}

In this appendix we establish 
$B(2-\frac{2}{N},N)\rightarrow
W^{2,N}(\mathbb{R}^{N+2})$ bounds for the extension operator $E$
defined in \eqref{def-extension}.  Such bounds are due to
E.~M. Stein (see \cite{S1,S2}).  We present the bounds here for
completeness, noting, in particular, that because more than 
one derivative is involved, one must impose a vanishing moment 
condition on the approximate identity $\phi$.

Define 
\begin{equation}\label{def-extension}
(Ef)(x,\xi)=\eta(\xi)\int_{\mathbb{R}^N} f(x-|\xi|y)\phi(y)\,dy
\end{equation}
for $x\in\mathbb{R}^N$ 
and $\xi\in\mathbb{R}^2$, where $\eta\in C_c^\infty(\mathbb{R}^2)$ and $\phi\in C_c^\infty(\mathbb{R}^N)$  are given such that
\begin{enumerate}
\item $0\leq \eta(\xi)\leq 1$ for all $\xi\in\mathbb{R}^2$, 
\item $\supp \eta\subset B(0;1)$, $\eta(0)=1$, and
\item $\int \phi \,dx=1$, $\supp \phi\subset B(0;1)$ and $\int x\phi(x)\,dx=0$.
\end{enumerate}

To simplify notation, we will write $(x,\xi)\in\mathbb{R}^N\times\mathbb{R}^2\cong \mathbb{R}^{N+2}$ with the conventions $x_{N+1}=\xi_1$ and $x_{N+2}=\xi_2$.  We obtain the following estimates:

\begin{proposition}
\label{prop-extension}
Fix $N\geq 3$ and, for each $f\in C^\infty_c(\mathbb{R}^N)$, let $Ef$ be defined by \eqref{def-extension}.  Then there exists $C>0$ such that 
\begin{align*}
\max_{|\alpha|\leq 2}\,\, \lVert D^\alpha Ef \rVert_{L^N(\mathbb{R}^{N+2})}&\leq C\lVert f\rVert_{2-\frac{2}{N},N}.
\end{align*}
where the maximum is taken over all multi-indices $\alpha$ with $|\alpha|\leq 2$.
\end{proposition}

\begin{proof}
In what follows, we give estimates for the quantities $\lVert Ef\rVert_{L^N(\mathbb{R}^{N+2})}$ and $\lVert \partial_{x_i}^2Ef\rVert_{L^N(\mathbb{R}^{N+2})}$ for $i=1,\cdots,N+2$.  Inspection of the arguments shows that the estimates for all $D^\alpha Ef$, $|\alpha|\leq 2$, follow from identical considerations.  Alternatively, one can appeal to standard interpolation inequalities and Calder\'on-Zygmund theory to show that the estimates on the zeroth and pure second derivatives of $Ef$ suffice to establish the claim.

To estimate the $L^N$ norm of $Ef$ itself, we apply Minkowski's integral inequality and Fubini's theorem to obtain
\begin{align}
\nonumber \lVert Ef\rVert_{L^N}&=\bigg\lVert \eta(\xi)\int_{\mathbb{R}^N} f(x-|\xi|y)\phi(y)\,dy\bigg\rVert_{L^N(\mathbb{R}^{N+2})}\\
&\leq \lVert \eta\rVert_{L^N(\mathbb{R}^2)}\lVert f\rVert_{L^N(\mathbb{R}^N)}\lVert \phi\rVert_{L^1(\mathbb{R}^N)}\label{eq-ef-bd1}
\end{align}

We now proceed to the pure second derivative estimates.  Fix $i\in \{1,2,\cdots,N\}$.  Then, for every $x\in\mathbb{R}^N$ and $\xi\in\mathbb{R}^2$,
\begin{align}
\nonumber \partial_i^2Ef(x,\xi)&=\eta(\xi)\int_{\mathbb{R}^N} (\partial_i^2f)(x-|\xi|y)\phi(y)\,dy\\
&=\frac{\eta(\xi)}{|\xi|}\int_{\mathbb{R}^N} (\partial_if)(x-|\xi|y)(\partial_i\phi)(y)\,dy\label{eq-diff}
\end{align}

Since $\phi$ has compact support, integration by parts gives $\int \partial_i\phi \,dy=0$.  We therefore obtain (with $f_i=\partial_if$, $\phi_i=\partial_i\phi$)
\begin{align}
\nonumber |\partial_i^2Ef(x,\xi)|&\leq \frac{1}{|\xi|}\int_{\mathbb{R}^N} |(f_i(x-|\xi|y)-f_i(x))\phi_i(y)|\,dy\\
\nonumber &=\frac{1}{|\xi|^{N+1}}\int_{\mathbb{R}^N} |(f_i(x-y)-f_i(x)) \phi_i(y/|\xi|)|\,dy\\
\nonumber &\leq \frac{1}{|\xi|^{1+N}}\lVert f_i(x-y)-f_i(x)\rVert_{L^N(\{y:|y|<|\xi|\})}\lVert \phi_i(\tfrac{\cdot}{|\xi|})\rVert_{L^{\frac{N}{N-1}}}\\
&\leq \frac{C}{|\xi|^2}\lVert f_i(x-y)-f_i(x)\rVert_{L^N(\{y:|y|<|\xi|\})}.\label{eq-fubini-1}
\end{align}
Integrating $|\partial_i^2Ef|^N$ with respect to $x$ and $\xi$, an application of Fubini's theorem gives
\begin{align}
\nonumber &\lVert \partial_i^2Ef\rVert_{L^N(\mathbb{R}^{N+2})}^N\\
\nonumber &\hspace{0.4in}\leq \int_{\mathbb{R}^N}\int_{\mathbb{R}^N} \bigg(\int_{|\xi|\geq |y|} |\xi|^{-2N}d\xi\bigg) |f_i(x-y)-f_i(x)|^{N} \,dxdy\\
&\hspace{0.4in}\leq C\lVert f\rVert_{2-\frac{2}{N},N}^N.\label{eq-fubini-2}
\end{align}

It remains to bound the norm of $\partial_{\xi_i}^2Ef$ for $i=1,2$.  By abuse of notation, we will abbreviate $\frac{\partial}{\partial\xi_\alpha}$ for $\alpha=1$ and $\alpha=2$ by $\partial_\alpha$.  Let $\alpha\in \{1,2\}$ be given.  We then have
\begin{align*}
\partial_\alpha^2Ef(x,\xi)&=\sum_{j=1}^4 (E)_j
\end{align*}
where we have set
\begin{align*}
(E)_1&:=\partial_\alpha^2\eta(\xi)\int f(x-|\xi|y)\phi(y)\,dy,\\
(E)_2&:=\frac{2\xi_\alpha\partial_\alpha\eta(\xi)}{|\xi|}\sum_{j=1}^N\int (\partial_jf)(x-|\xi|y)\phi(y)y_j\,dy,\\\
(E)_3&:=\frac{(|\xi|^2-(\xi_\alpha)^2)\eta(\xi)}{|\xi|^3}\sum_{j=1}^N\int (\partial_jf)(x-|\xi|y)\phi(y)y_j\,dy,\\
(E)_4&:=\frac{(\xi_\alpha)^2\eta(\xi)}{|\xi|^2}\sum_{j=1}^N\sum_{k=1}^N\int (\partial_{j,k}f)(x-|\xi|y)\phi(y)y_jy_k\,dy.
\end{align*}

The contributions of $(E)_1$ and $(E)_2$ are estimated as in \eqref{eq-ef-bd1} above, giving the bounds
\begin{align*}
\lVert (E)_1\rVert_{L^N(\mathbb{R}^{N+2})}&\leq \lVert D^2\eta\rVert_{L^N(\mathbb{R}^2)}\lVert f\rVert_{L^N(\mathbb{R}^N)}\lVert \phi\rVert_{L^1(\mathbb{R}^N)}
\end{align*}
and
\begin{align*}
\lVert (E)_2\rVert_{L^N(\mathbb{R}^{N+2})}&\leq C\lVert \nabla\eta\rVert_{L^N(\mathbb{R}^2)}\lVert \nabla f\rVert_{L^N(\mathbb{R}^N)}\lVert y\phi(y)\rVert_{L^1(\mathbb{R}^N)},
\end{align*}
respectively.  Turning to the contribution of $(E)_3$, we write
\begin{align}
\lVert (E)_3\rVert_{L^N(\mathbb{R}^{N+2})}&\leq \sum_{j=1}^N\bigg\lVert \frac{\eta(\xi)}{|\xi|}\int (\partial_jf)(x-|\xi|y)\phi(y)y_j\,dy\bigg\rVert_{L^N(\mathbb{R}^{N+2})}\label{eq-e3-rhs1}
\end{align}
so that by making use of the moment condition $\int y\phi(y)\,dy=0$ and arguing as in \eqref{eq-fubini-1}--\eqref{eq-fubini-2} we get
\begin{align}
\eqref{eq-e3-rhs1}\leq C\bigg\lVert \int \frac{(\partial_jf)(x-|\xi|y)-(\partial_jf)(x)}{|\xi|}y\phi(y)\,dy\bigg\rVert_{L^N(\mathbb{R}^{N+2})}\leq C\lVert f\rVert_{2-\frac{2}{N},N}\label{eq-e3-bd}
\end{align}

Finally,
\begin{align}
\lVert (E)_4\rVert_{L^N(\mathbb{R}^{N+2})}&=\bigg\lVert \frac{\eta(\xi)}{|\xi|}\sum_{j,k} \int (\partial_jf)(x-|\xi|y)\Phi_{j,k}(y)\,dy\bigg\rVert_{L^N(\mathbb{R}^{N+2})}\label{eq-e4}
\end{align}
where we have set $\Phi_{j,k}(y)=(\partial_k[y_jy_k\phi])(y)$ for $y\in\mathbb{R}^N$.  Noting that $\int \Phi_{j,k}(y)\,dy=0$, we obtain
\begin{align}
\eqref{eq-e4}&\leq C\bigg\lVert \int \frac{(\partial_jf)(x-|\xi|y)-(\partial_jf)(x)}{|\xi|}\Phi_{j,k}(y)\,dy\bigg\rVert_{L^N(\mathbb{R}^{N+2})}\label{eq-e4-2}
\end{align}
so that another application of the estimates in \eqref{eq-fubini-1}--\eqref{eq-fubini-2} gives the bound
\begin{align}
\eqref{eq-e4-2}&\leq C\lVert f\rVert_{2-\frac{2}{N},N}.\label{eq-fubini-3}
\end{align}

Combining the above estimates, we have shown 
\begin{align*}
\lVert \partial_{\alpha}^2Ef\rVert_{L^N(\mathbb{R}^{N+2})}\leq C\lVert f\rVert_{2-\frac{2}{N},N}.
\end{align*}
When taken together with the estimates $\lVert \partial_i^2Ef\rVert_{L^N}\leq C\lVert f\rVert_{2-\frac{2}{N},N}$ for $i=1,\cdots,N$, and in view of the remarks at the beginning of the proof, this completes the proof of Proposition \ref{prop-extension}.
\end{proof}

\section{The three-dimensional case of Proposition \ref{prop-Hbound-3}.}

In this appendix, we give a different, more direct proof of Proposition 
$\ref{prop-Hbound-3}$ in the three-dimensional case.  

%In this setting, 
%combinatorial considerations allow for a more direct proof in which the 
%contribution of each term in the relevant determinant calculation can 
%be estimated individually.

\begin{proposition}
Fix $N=3$.  Let the sequences $(C_{{\boldsymbol\ell}})$ and $(H_{{\boldsymbol\ell}})$ be as in \eqref{eq-def-C}--\eqref{eq-def-D}.  Then there exists $C>0$ such that
for all ${\boldsymbol\ell}\in\mathscr{L}$, we have
\begin{align}
C_{{\boldsymbol\ell}}\bigg|\int \det(H_{{\boldsymbol \ell}})(x)\varphi(x)\,dx\bigg|
\leq \frac{C\lVert \varphi\rVert_{C^1}}{k^6}.\label{eq-3d-goal}
\end{align}
\end{proposition}

\begin{proof}[Proof of \eqref{eq-3d-goal}] \ 
Let $(\ell_1,\ell_2,\ell_3)\in \mathscr{L}$ be given, and suppose, without loss of generality, that $\lVert \varphi\rVert_{C^1}\leq 1$.  Because
$\partial_{3,3}f_\ell=0$ for all $1\leq \ell\leq k$, 
we may write the left-hand side of \eqref{eq-3d-goal} as
\begin{align}
\sum_{\substack{\sigma\in S_3,\\\sigma(3)\neq 3}} \frac{1}{(n_{\ell_1}n_{\ell_2}n_{\ell_3})^{4/3}(\ell_1\ell_2\ell_3)^{1/3}}\Big|\Big\langle h_{\sigma},\varphi\Big\rangle\Big|\label{eq-3d-1}
\end{align}
with
\begin{align*}
h_{\sigma}&=(\partial_{1,\sigma(1)}f_{\ell_{\sigma(1)}})(\partial_{2,\sigma(2)}f_{\ell_{\sigma(2)}})(\partial_{3,\sigma(3)}f_{\ell_{\sigma(3)}}),
\end{align*}

We estimate each term of the sum \eqref{eq-3d-1}.  Let $\sigma\in S_3$ be given with $\sigma(3)\neq 3$.  Recalling \eqref{def-f-ell}, we obtain the following identities by direct computation:
\begin{itemize}
\item[(i)] for $1\leq i,j\leq 2$ and $i\neq j$,
\begin{align*}
(\partial_{i,j}f_{\ell_j})&=n_{\ell_j}^2x_3\sin(2n_{\ell_j} x_1)\sin(2n_{\ell_j} x_2)
\end{align*}
\item[(ii)] for $1\leq i,j\leq 2$ and $i=j$,
\begin{align*}
(\partial_{i,j}f_{\ell_j})&=2n_{\ell_j}^2x_3\cos(2n_{\ell_j}x_i)\sin^2(n_{\ell_j}x_{i'})
\end{align*}
where $i'$ is the single element of $\{1,2\}\setminus \{i\}$,
\item[(iii)] for $1\leq i\leq 2$, $j=3$,
\begin{align*}
(\partial_{i,j}f_{\ell_j})&=n_{\ell_j}\sin(2n_{\ell_j}x_i)\sin^2(n_{\ell_j}x_{i'})
\end{align*}
where $i'$ is the single element of $\{1,2\}\setminus \{i\}$,
and
\item[(iv)] for $i=3$, $1\leq j\leq 2$,
\begin{align*}
(\partial_{i,j}f_{\ell_j})&=n_{\ell_j}\sin(2n_{\ell_j}x_j)\sin^2(n_{\ell_j}x_{j'}).
\end{align*}
where $j'$ is the single element of $\{1,2\}\setminus \{j\}$.
\end{itemize}

Let $j$ denote the single element of $\{1,2\}\setminus \{\sigma(3)\}$.  In view of the above identities, and recalling that $\varphi(x)$ has the form $\prod_{i=1}^3 \varphi_i(x_i)$ by hypothesis (with support contained in the cube $(0,2\pi)^N$), we conclude that there exist trigonometric polynomials $P_1$ and $P_2$ satisfying
\begin{align}
\sup_{x_1\in [0,2\pi]} |P_1(x_1)|\leq 1, \quad \sup_{x_2\in [0,2\pi]} |P_2(x_2)|\leq 1\label{eq-p1p2}
\end{align}
and chosen so that the term in \eqref{eq-3d-1} corresponding to the permutation $\sigma$ is bounded by
\begin{align}
\nonumber &\frac{n_{\ell_{j}}^{2/3}}{\displaystyle(n_{\ell_3}n_{\ell_{\sigma(3)}})^{1/3}(\ell_1\ell_2\ell_3)^{1/3}}\bigg|\bigg(\int_{\mathbb{R}} P_1(x_1)\varphi_1(x_1)\,dx_1\bigg)\\
&\hspace{1.8in} \cdot\bigg(\int_{\mathbb{R}} P_2(x_2)\varphi_2(x_2)\,dx_2\bigg)\bigg|\lVert \varphi_3\rVert_{L^\infty}.\label{eq-14}
\end{align}

Set $\ell_*=\max\{\ell_1,\ell_2,\ell_3\}$.  We consider several cases:

\vspace{0.1in}
\underline{Case 1}: $\ell_{j}\neq \ell_*$.
\vspace{0.1in}

In this case, the bounds in \eqref{eq-p1p2} imply that \eqref{eq-14} is bounded by
\begin{align*}
\frac{n_{\ell_j}^{2/3}}{(n_{\ell_3}n_{\ell_{\sigma(3)}})^{1/3}}&\leq \frac{n_{\ell_*-1}^{2/3}}{n_{\ell_*}^{1/3}}=k^\gamma
\end{align*}
with
\begin{align*}
\gamma=\frac{2}{3}(3^{3(\ell_*-1)})-\frac{1}{3}(3^{3\ell_*})=-25(3^{3\ell_*-4})\leq -6,
\end{align*}
where to obtain the last inequality we have recalled that $\ell_*\geq 2$ holds by construction.  This completes the proof of the proposition in this case.

\vspace{0.1in}
\underline{Case 2}: $\ell_{j}=\ell_*$ and $\sigma^{-1}(j)\neq j$.
\vspace{0.1in}

Set $i=\sigma^{-1}(j)$ and note that the condition $i\neq 3$ follows from the definition of $j$.  It now follows from 
$\sigma(3)\neq j$ (and $j\neq i$) that $\sigma(3)=i$, and consequently $\sigma(j)=3$.

Computing the factors of $h_{\sigma}$, we have
\begin{align*}
\partial_{i,\sigma(i)}f_{\ell_{\sigma(i)}}&=\partial_{i,j}f_{\ell_{j}}=n_{\ell_{j}}^2x_3\sin(2n_{\ell_{j}}x_1)\sin(2n_{\ell_{j}}x_2),\\
\partial_{j,\sigma(j)}f_{\ell_{\sigma(j)}}&=\partial_{j,3}f_{\ell_3}=n_{\ell_3}\sin(2n_{\ell_3}x_j)\sin^2(n_{\ell_3}x_i)
\end{align*}
and
\begin{align*}
\partial_{3,\sigma(3)}f_{\ell_{\sigma(3)}}&=\partial_{3,i}f_{\ell_i}=n_{\ell_{i}}\sin(2n_{\ell_{i}}x_{i})\sin^2(n_{\ell_{i}}x_{j}).
\end{align*}

Note that $(\ell_1,\ell_2,\ell_3)\in\mathscr{L}$ implies that at least one of the conditions $\ell_i\neq \ell_*$ or $\ell_3\neq \ell_*$ holds.  Suppose first
that both of these conditions hold, i.e. $\ell_i\neq \ell_*$ and $\ell_3\neq \ell_*$.  In this case, we note that $P_i$ is given by 
\begin{align*}
P_i(x_i)=\sin(2n_{\ell_*}x_i)\sin(2n_{\ell_i}x_i)\sin^2(n_{\ell_3}x_i).
\end{align*}
Integrating by parts with respect to $x_i$ in \eqref{eq-14}, it follows that \eqref{eq-14} is bounded by a multiple of
\begin{align}
\frac{n_{\ell_i}^{2/3}}{(n_{\ell_3}n_{\ell_*})^{1/3}}+\frac{n_{\ell_3}^{2/3}}{(n_{\ell_i}n_{\ell_*})^{1/3}} \label{eq-3d-2}
\end{align}

Arguing as in Case $1$ above, we obtain the bound
\begin{align*}
\eqref{eq-3d-2}&\leq k^{-6}
\end{align*}
as desired.

Suppose now that $\ell_i\neq \ell_*$ and $\ell_3=\ell_*$.  In this case, we have
\begin{align*}
P_i(x_i)&=\sin(2n_{\ell_*}x_i)\sin(2n_{\ell_i}x_i)\sin^2(n_{\ell_*}x_i).
\end{align*}
In view of the identity 
\begin{align}
\sin(2n_{\ell_*}x_i)\sin^2(n_{\ell_*}x_i)&=\frac{1}{2n_{\ell_*}}\frac{d}{dx_i}[\sin^4(n_{\ell_*}x_i)],\label{trig-identity}
\end{align}
integration by parts shows that \eqref{eq-14} is bounded by a multiple of
\begin{align}
\frac{n_{\ell_i}^{2/3}}{n_{\ell_*}^{2/3}}\leq \frac{n_{\ell_*-1}^{2/3}}{n_{\ell_*}^{2/3}}=k^{\gamma},\quad \gamma=\frac{2}{3}(3^{3\ell_*-3}-3^{3\ell_*}).\label{eq-3d-3}
\end{align}
The desired conclusion ($\gamma\leq -6$) then follows again from $\ell_*\geq 2$.

Similarly, if $\ell_i=\ell_*$ and $\ell_3\neq \ell_*$, we have
\begin{align*}
P_j(x_j)&=\sin(2n_{\ell_*}x_j)\sin^2(n_{\ell_*}x_j)\sin(2n_{\ell_3}x_j)
\end{align*}
and thus, using integration by parts as above, \eqref{eq-14} is bounded by a multiple of
\begin{align}
(n_{\ell_3}^{2/3})/(n_{\ell_*}^{2/3}),\label{eq-3d-4}
\end{align}
which is again bounded by $k^{-6}$ as desired.

\vspace{0.1in}
\underline{Case 3}: $\ell_{j}=\ell_*$ and $\sigma^{-1}(j)=j$.
\vspace{0.1in}

Set $i=\sigma(3)$, and observe that $i,j\in \{1,2\}$ with $i\neq j$.  It then follows from $\sigma(j)=j$ and $\sigma(3)\neq j$ that $\sigma(j)=3$ and $\sigma(3)=i$.  As in our treatment of Case $2$, we therefore identify the factors of $h_{\sigma}$ as
\begin{align}
\partial_{j,\sigma(j)}f_{\ell_{\sigma(j)}}&=\partial_{j,j}f_{\ell_{j}}=n_{\ell_j}x_3\cos(2n_{\ell_j}x_j)\sin^2(n_{\ell_j}x_j),\label{eq-3-4}\\
\partial_{i,\sigma(i)}f_{\ell_{\sigma(i)}}&=\partial_{i,3}f_{\ell_{3}}=n_{\ell_3}\sin(2n_{\ell_3}x_i)\sin^2(n_{\ell_3}x_j),\label{eq-3-5}
\end{align}
and
\begin{align}
\partial_{3,\sigma(3)}f_{\ell_{\sigma(3)}}&=\partial_{3,i}f_{\ell_{i}}=n_{\ell_i}\sin(2n_{\ell_i}x_i)\sin^2(n_{\ell_i}x_j).\label{eq-3-6}
\end{align}

We again observe that $(\ell_1,\ell_2,\ell_3)\in\mathscr{L}$ implies that at least one of the frequencies $\ell_i$ or $\ell_3$ is not equal to $\ell_*$.  If $\ell_i\neq \ell_*$ and $\ell_3\neq \ell_*$, we use \eqref{eq-3-4}--\eqref{eq-3-6} to write
\begin{align*}
P_j=\cos(2n_{\ell_*}x_j)\sin^2(n_{\ell_3}x_j)\sin^2(n_{\ell_i}x_j)
\end{align*}
Applying integration by parts, this bounds \eqref{eq-14} by a multiple of
\begin{align}
\max\{n_{\ell_3}^{2/3},n_{\ell_i}^{2/3}\}/(n_{\ell_*}^{1/3}).\label{eq-3-7}
\end{align}

In the case $\ell_i\neq \ell_*$, $\ell_3=\ell_*$, we similarly have
\begin{align*}
P_{i}&=\sin^2(n_{\ell_*}x_i)\sin(2n_{\ell_*}x_i)\sin(2n_{\ell_i}x_i)
\end{align*}
which (in view of \eqref{trig-identity}) gives the bound
\begin{align}
\eqref{eq-14}&\leq C(n_{\ell_i}^{2/3})/(n_{\ell_*}^{1/3}),\label{eq-3-8}
\end{align}
while for $\ell_i=\ell_*$, $\ell_3\neq \ell_*$, we have
\begin{align*}
P_{i}&=\sin^2(n_{\ell_*}x_i)\sin(2n_{\ell_3}x_i)\sin(2n_{\ell_*}x_i),
\end{align*}
giving the bound
\begin{align}
\eqref{eq-14}&\leq C(n_{\ell_3}^{2/3})/(n_{\ell_*}^{1/3}).\label{eq-3-9}
\end{align}

Arguing as in Cases $1$ and $2$ above, each of the bounds \eqref{eq-3-7}--\eqref{eq-3-9} are controlled by $k^{-6}$.  This completes the proof. 
\end{proof}

\end{document}